\documentclass[11pt,a4paper]{elsarticle}

\usepackage[utf8]{inputenc}

\usepackage{amsmath}
\usepackage{amsthm}
\usepackage{enumerate}
\usepackage{mathdots}
\usepackage{amsfonts}
\usepackage{multicol}
\usepackage{amsmath}
\usepackage{amssymb}
\usepackage{fancybox}
\usepackage[usenames,dvipsnames]{color}
\usepackage{graphicx}
\usepackage{tikz}
\usepackage[left=3cm,top=4cm,right=3cm,bottom=3.5cm]{geometry}
\usepackage{pifont}
\usepackage{natbib}
\usepackage{geometry}
\usepackage{mathtools}

\makeatletter
\def\bbordermatrix#1{\begingroup \m@th
	\global\let\perhaps@scriptstyle\scriptstyle
	\@tempdima 4.75\p@
	\setbox\z@\vbox{%
		\def\cr{%
			\crcr
			\noalign{%
				\kern2\p@
				\global\let\cr\endline
				\global\let\perhaps@scriptstyle\relax
			}%
		}%
		\ialign{$\make@scriptstyle{##}$\hfil\kern2\p@\kern\@tempdima
			&\thinspace\hfil$\perhaps@scriptstyle##$\hfil
			&&\quad\hfil$\perhaps@scriptstyle##$\hfil\crcr
			\omit\strut\hfil\crcr
			\noalign{\kern-\baselineskip}%
			#1\crcr\omit\strut\cr}}%
	\setbox\tw@\vbox{\unvcopy\z@\global\setbox\@ne\lastbox}%
	\setbox\tw@\hbox{\unhbox\@ne\unskip\global\setbox\@ne\lastbox}%
	\setbox\tw@\hbox{$\kern\wd\@ne\kern-\@tempdima\left[\kern-\wd\@ne
		\global\setbox\@ne\vbox{\box\@ne\kern2\p@}%
		\vcenter{\kern-\ht\@ne\unvbox\z@\kern-\baselineskip}\,\right]$}%
	\null\;\vbox{\kern\ht\@ne\box\tw@}\endgroup}
\def\make@scriptstyle#1{\vcenter{\hbox{$\scriptstyle#1$}}}
\makeatother

\newcommand{\C}{\mathbb{C}}

\newcommand{\efe}{\mathbb{F}}
\newcommand{\FF}{\mathbb{F}}
\newcommand{\F}{\mathbb{F}}

\newcommand{\la}{\lambda}

\def\rank{\mathop{\rm rank}\nolimits}
\newcommand{\wh}{\widehat}

\newtheorem{theo}{Theorem}[section]

\newtheorem{deff}[theo]{Definition}

\newtheorem{prop}[theo]{Proposition}

\newtheorem{lem}[theo]{Lemma}

\newtheorem{rem}[theo]{Remark}

\newtheorem{example}[theo]{Example}

\DeclareMathOperator{\diag}{diag}
\DeclareMathOperator{\rev}{rev}

\begin{document}
\title{Linearizations of rational matrices from general representations}

\begin{abstract}
We construct a new family of linearizations of rational matrices $R(\la)$  written in the general form  $R(\lambda)= D(\la)+C(\la)A(\la)^{-1}B(\la)$, where $D(\la)$, $C(\la)$, $B(\la)$ and $A(\la)$ are polynomial matrices. 
 Such representation always exists and are not unique. The new linearizations are constructed from linearizations of the polynomial matrices $D(\la)$ and $A(\la)$, where each of them can be represented in terms of any polynomial basis.
In addition, we show how to recover eigenvectors, when $R(\lambda)$ is regular, and  minimal bases and minimal indices, when $R(\lambda)$ is singular, from those of their linearizations in this family. 
\end{abstract}


\date{ }
\author[montana]{Javier Pérez\corref{cor1}}
\ead{javier.perez-alvaro@mso.umt.edu}

\author[uc3m]{María C. Quintana\fnref{fn2}}
\ead{maquinta@math.uc3m.es}

\cortext[cor1]{Corresponding author}

\address[montana]{Department of Mathematical Sciences, University of Montana, USA.}

\address[uc3m]{Departamento de Matem\'aticas,
	Universidad Carlos III de Madrid, Avda. Universidad 30, 28911 Legan\'es, Spain.\\
\vspace{0.5cm}
{\rm Dedicated to Paul Van Dooren on the occasion of his 70th birthday }}

\fntext[fn2]{Supported by ``Ministerio de Econom\'ia, Industria y Competitividad (MINECO)'' of Spain and ``Fondo Europeo de Desarrollo Regional (FEDER)'' of EU through grants MTM2015-65798-P and MTM2017-90682-REDT, and the predoctoral contract BES-2016-076744 of MINECO.}

\maketitle

\section{Introduction}

Let $R(\la)$ be a rational matrix, that is, a matrix whose entries are scalar rational functions in the variable $\la$.
 The rational eigenvalue problem (REP) consists of finding scalars $\la_0$ such that $\la_0$ is not a pole of $R(\la)$, i.e., $R(\la_0)$ has finite entries, and that there exist nonzero constant vectors $x$ and $y$ satisfying 
\[
R(\la_0)x=0 \quad\mbox{and}\quad y^{T}R(\la_0)=0,
\]
under the assumption that $R(\la)$ is regular, i.e., $R(\lambda)$ is square and its determinant is not identically equal to zero.
 The scalar $\la_0$ is said to be an eigenvalue of $R(\la)$ and the vectors $x$ and $y^T$ are called, respectively, right and left eigenvectors associated to $\la_0$. 

A non-regular rational matrix $R(\lambda)$ is also called singular. 
 In general, regardless of whether $R(\lambda)$ is regular or singular, a scalar $\la_0$ is said to be an eigenvalue of $R(\la)$ if $\la_0$ is not a pole of $R(\la)$ and
\[  
 \rank R(\la_0)< \rank R(\la),
\]
where $\rank R(\la)$ denotes the rank of $R(\lambda)$ over the field of rational functions in $\lambda$.
 Moreover, the problem of finding the eigenvalues of a rational matrix can also be seen as the problem of finding the zeros of $R(\la)$ that are not poles. 
If $\la_0$ is a pole of $R(\la)$ and there exists a polynomial vector $v(\la)$ such that $v(\la_0)\neq 0$ and that $\lim_{\la \to \la_0}R(\la)v(\la)=0$ then $\la_0$ is said to be an eigenpole of $R(\la)$  \cite{AlBe16}. 

Applications of rational matrices in linear systems and control theory can be found in the classical references \cite{Kailath80,McMi2, Rosen70, Vard91}, and numerical algorithms for computing zeros and poles and other structural data of rational matrices can be found in \cite{vandooren1981} and the references therein.

The numerical solution of REPs is recently getting a lot of attention from the numerical linear algebra community since REPs appear directly from applications \cite{guttel-tisseur-2017, SuBai} or as approximations to arbitrary nonlinear eigenvalue problems (NLEPs) \cite{nlep,automatic,Saad}.

A competitive method for solving REPs is linearization \cite{AlBe16, strong, local, dopmarquin2019, SuBai}. Linearization transforms the REP into a generalized eigenvalue problem in such a way that the pole and zero information of the corresponding rational matrix is preserved. 
The generalized eigenvalue problem is then solved by the QZ algorithm, for small to medium size dense matrices, or with Krylov methods adapted to the particular structure of the linearization, for large-scale problems and sparse matrices \cite{compact_rat_Krylov,nlep}.

The construction of linearizations for rational matrices presented in \cite{AlBe16, strong,dopmarquin2019, SuBai} depends on two facts.
First, any rational matrix $R(\la)$ can be uniquely written as a sum of the form $R(\lambda)=D(\lambda)+R_{sp}(\lambda)$, where $D(\lambda)$ is a polynomial matrix and $R_{sp}(\lambda)$ is a strictly proper rational matrix \footnote{A strictly proper rational matrix is a matrix whose entries are strictly proper rational functions.
A rational function $p(\lambda)/q(\lambda)$ is strictly proper when the degree of the polynomial $p$ is smaller than the degree of the polynomial $q$.}.
And, second, the strictly proper rational matrix $R_{sp}(\lambda)$ can be written in the form 
\begin{equation}\label{eq:state-space-form}
R_{sp}(\lambda)=C( I_n \la -A)^{-1}B,
\end{equation}
where $A$, $B$ and $C$ are constant matrices.
The decomposition \eqref{eq:state-space-form} is called a state-space realization of $R_{sp}(\lambda)$, and can always be taken to be  minimal (i.e., of least order); see, for example, \cite[Ch. 3, Sec. 5.1]{Rosen70} or \cite[Sec. 1.10]{Vard91}.


Representations of the form 
\begin{equation}\label{eq:state-space representation}
R(\lambda) = D(\la)+C(I_n \la  -A)^{-1}B
\end{equation}
arise naturally in control problems.
Furthermore, if the rational  matrix $R(\lambda)$ is not in the form \eqref{eq:state-space representation}, there exist procedures for obtaining such a representation \cite{Kailath80, Rosen70, Vard91}.
However, these procedures are not simple, and may introduce errors that were not present in the original problem.
For this reason, the main aim of this work is to construct linearizations for rational matrices from more general representations. 

We will show how to construct linearizations of rational matrices that are written in the general form
\begin{equation}\label{eq:general representation}
R(\lambda) = D(\la)+C(\la)A(\la)^{-1}B(\la),
\end{equation}
where $A(\la)$, $B(\la)$, $C(\la)$ and $D(\la)$ are polynomial matrices, possibly non linear and possibly expressed in different bases. 
Representations of the form \eqref{eq:general representation} arise naturally, for example, when solving nonlinear eigenvalue problems of the form
\[
R(\lambda)x=\left(P(\la)+\displaystyle\sum_{i=1}^{m}\frac{n_{i}(\la)}{d_{i}(\la)}( A_{i} \la -B_i)\right)x=0,
\]
where $P(\la)$ is a polynomial matrix, $\frac{n_{i}(\la)}{d_{i}(\la)}$ are scalar rational functions, and $A_i$ and $B_i$ are constant matrices.

 The paper is organized as follows.
After a preliminary section, where we introduce the notation used throughout the paper and some basic results, we present the new family of  linearizations for rational matrices in Sections \ref{sec:local} and \ref{sec:infinity}.
 In Section \ref{example}, we present an example that highlights the difference between our approach and the previous approaches to the problem of linearizing rational matrices.
 In Section \ref{sec:recovery} we study how to recover minimal bases, minimal indices and eigenvectors of rational matrices from those of their linearizations constructed in Sections \ref{sec:local} and \ref{sec:infinity}. Finally, we apply the new linearizations to solve scalar rational equations in Section \ref{sec:scalar}, and we give some concluding remarks in Section \ref{conclusion}.
 
\section{Preliminaries}\label{prelim}

Throughout this work, we denote by $\efe$ an algebraically closed field that does not include infinity.

\subsection{Polynomial and rational matrices}
Let $\efe[\la]$ denote the ring of polynomials with coefficients in $\efe$, and let $\efe(\la)$ denote the field of rational functions.
By $\efe^{p\times m}$, $\efe[\la]^{p\times m}$ and $\efe(\la)^{p\times m}$, we denote the sets of $p\times m$ matrices with entries in $\efe,$ $\efe[\la]$ and $\efe(\la),$  respectively. 
The elements of $\efe[\la]^{p\times m}$ are called polynomial matrices, and the elements of $\efe(\la)^{p\times m}$ are called rational matrices. 

A rational matrix $R(\la)$ is said to be regular if it is invertible for some $\la_0\in\F$.
In words, a rational matrix $R(\lambda)$ is regular if it is square and $\det R(\lambda)$ is non-zero for some $\la_0\in\F$.
A non-regular rational matrix $R(\lambda)$ is called singular.
If a rational matrix $R(\lambda)$ is invertible for all $\lambda$ in some set $\Omega\subseteq \mathbb{F}$, we say that $R(\lambda)$ is regular in $\Omega$.

We denote by $\rank R(\lambda)$ the rank of the rational matrix $R(\lambda)$ over the field $\mathbb{F}(\lambda)$.
Abusing notation, if $\lambda_0\in\efe$, we denote by $\rank R(\lambda_0)$ the rank over the field $\efe$ of the constant matrix $R(\lambda_0)$.

A polynomial matrix $P(\lambda)\in\efe[\lambda]^{p\times m}$ can always be written in the form
\begin{equation}\label{eq:matrix poly}
P(\lambda) =  P_d \lambda^d +  P_{d-1} \lambda^{d-1} + \cdots + P_1  \lambda + P_0, 
\end{equation}
for some constant matrices  $P_d,\hdots,P_1,P_0\in\mathbb{F}^{p\times m}$, with $P_d\neq 0$. 
The scalar $d$ is called the degree of $P(\lambda)$, and it is denoted by $d=\deg P(\lambda)$.
If the leading coefficient matrix $P_d=0$, then $d$ is called the grade of $P(\lambda)$.
By $\mathrm{grade}\, P(\lambda)$, we denote the grade of a polynomial matrix $P(\lambda)$.


A square polynomial matrix whose determinant is a nonzero constant is called unimodular.
If $V(\lambda)$ is unimodular, its inverse $V(\lambda)^{-1}$ is also unimodular.
Two polynomial matrices $P(\lambda)$ and $Q(\lambda)$ are said to be unimodularly equivalent if there are unimodular polynomial matrices $U(\lambda)$ and $V(\lambda)$ such that $P(\lambda) = U(\lambda)Q(\lambda)V(\lambda)$.

We now review some concepts and ideas from system theory; for more details see, for example, \cite{Rosen70}. 
It is known that any rational matrix $R(\la)\in\FF(\la)^{p\times m}$, regular or singular, can be written as
\begin{equation}\label{eq_realizR}
R(\la)=D(\la)+C(\la)A( \la)^{-1}B(\la)
\end{equation}
for some polynomial matrices  $A(\la)\in\FF[\la]^ {n\times n}$,
$B(\la)\in\FF[\la]^{n\times m}$, $C(\la)\in\FF[\la]^{p\times n}$ and $D(\la)\in\FF[\la]^{p\times m},$ where  $A(\la)$ is regular if $n>0$.
The case $n=0$ means that $R(\la)$ is a polynomial matrix, but polynomial matrices can be also be represented as in \eqref{eq_realizR} with $n>0$.
An expression of the form \eqref{eq_realizR} is called a realization of $R(\lambda)$.
We recall that realizations are not unique.

The polynomial matrix
\begin{equation}\label{eq_polsysmat}
P(\la)=\begin{bmatrix}
A(\la) & B(\la)\\
-C(\la) & D(\la)
\end{bmatrix}
\end{equation}
is said to be a polynomial system matrix of the rational matrix $R(\la)$ in \eqref{eq_realizR}.
The rational matrix $R(\la)$ is called the transfer function matrix of $P(\la),$ that is, $R(\la)$ is the Schur complement of $A(\la)$ in $P(\la)$. 
The matrix $A(\la)$ is called the state matrix of $P(\la).$ 
We notice that $P(\lambda)$ admits the block LDU factorization 
	\begin{equation*}\label{factorizationLDU}
P(\la)=\begin{bmatrix}
I_n & 0\\
-C(\la)A(\la)^{-1} & I_p
\end{bmatrix}\begin{bmatrix}
A(\la) & 0\\
0 & R(\la)
\end{bmatrix}\begin{bmatrix}
I_n & A(\la)^{-1}B(\la)\\
0 & I_m
\end{bmatrix}.
	\end{equation*}
Since $A(\la)$ is regular, the following rank property is satisfied:
\begin{equation}\label{ranks_rel}
\rank P(\la)= n + \rank R(\la).
\end{equation}

The polynomial system matrix $P(\la)$ in \eqref{eq_polsysmat}, with $n>0$, is said to have least order, or to be minimal, if the matrices 
\begin{equation} \label{mat_minimalidad}
 \begin{bmatrix} A(\la) \\ C(\la)\end{bmatrix}\quad \text{and} \quad 	\begin{bmatrix} A(\la) & B(\la) \end{bmatrix}
 \end{equation}
 have no eigenvalues in $\mathbb{F}$.
 In such case, we also say that the realization in \eqref{eq_realizR} is minimal.
  If the matrices in \eqref{mat_minimalidad} have no eigenvalues in a subset $\Omega\subseteq\mathbb{F}$ then $P(\la)$ is said to be minimal in $\Omega$ \cite[Definition 3.3]{local}. 
In such case, we also say that the realization \eqref{eq_realizR} is minimal in $\Omega.$ The main property of having minimality is that if the polynomial system matrix $P(\la)$ is minimal in $\Omega$, then the poles of $R(\la)$ in $\Omega$ are the zeros of the state matrix $A(\la)$ in $\Omega$, and the zeros of $R(\la)$ in $\Omega$ are the zeros of $P(\la)$ in $\Omega$ (see \cite[Theorem 3.7]{local}). Poles and zeros of a rational matrix are defined through the notion of the Smith-McMillan form \cite{McMi2}, other more recent references for the Smith-McMillan form of a rational matrix are \cite{Kailath80,Rosen70,Vard91}. In particular, the finite poles and zeros of a rational matrix are the roots in $\F$ of the polynomials that appear on the denominators and numerators, respectively, in its (finite) Smith-McMillan form. Poles and zeros at infinity of a rational matrix $R(\la)$ are defined as the poles and zeros at $\la=0$ of $R(1/\la)$ (see \cite{Kailath80}).

\subsection{Minimal bases and minimal indices}

In this section, we review the notions of minimal bases and minimal indices of rational subspaces \cite{forney}.

It is known that every rational vector subspace $\mathcal{V} \subseteq \mathbb{F}(\la)^n$ over the field $\mathbb{F}(\la)$ has bases consisting of polynomial vectors. 
We refer to such bases as polynomial bases.
Following \cite{forney}, a minimal  basis of $\mathcal{V}$ is a polynomial basis of $\mathcal{V}$ consisting of polynomial vectors whose sum of degrees is minimal among all polynomial bases of $\mathcal{V}$. 
Minimal bases are not unique, but the ordered list of degrees of the polynomial vectors in any minimal basis of $\mathcal{V}$ is always the same. 
Hence, these degrees are uniquely determined by $\mathcal{V}$ and are called the minimal indices of $\mathcal{V}$.

Let now $R(\la)\in\FF(\la)^{p\times m}$ be a rational matrix, we consider the rational vector subspaces:
\[
\begin{array}{l}
\mathcal{N}_r (R)=\{x(\la)\in\FF(\la)^{m\times 1}: R(\la)x(\la)=0\}, \text{ and}\\
\mathcal{N}_\ell (R)=\{y(\la)^{T}\in\FF(\la)^{1\times p}: y(\la)^T R(\la)=0\},

\end{array}
\]
which are called the right and left null-spaces of $R(\la)$, respectively. If $R(\lambda)$ is singular at least one of these null-spaces is non-trivial. If ${\cal N}_{r}(R)$ (resp. ${\cal N}_{\ell}(R)$) is non-trivial, it has minimal bases and minimal indices, which are called the right (resp. left) minimal bases and minimal indices of $R(\la)$. 
If $r=\rank R(\lambda)$, then $R(\la)$ has $p-r$ left minimal indices and $m-r$ right minimal indices.

Lemma \ref{Lemma: from R to P} establishes a linear map between the right (resp. left) nullspace of $R(\la)$ and the right (resp. left) nullspace of a polynomial system matrix $P(\la)$ of $R(\la)$.   
\begin{lem}\label{Lemma: from R to P} 
	Let 
	\[
	P(\la)=\begin{bmatrix}
	A(\la) & B(\la)\\
	-C(\la) & D(\la)
	\end{bmatrix}\in \FF[\la]^{(n+p)\times (n+m)}
	\]
	be a polynomial system matrix, with state matrix $A(\la)\in\efe[\lambda]^{n\times n},$ and transfer function matrix $R(\la) = D(\lambda)+C(\lambda)A(\lambda)^{-1}B(\lambda)\in \efe(\la)^{p\times m}$. Then, the following statements hold:
	\begin{itemize}
		\item[\rm(a)] The map
		\begin{align*}
		T_r \, : \, \mathcal{N}_r(R) & \longrightarrow \mathcal{N}_r(P) \\
		x(\lambda) & \longmapsto  \begin{bmatrix} -A(\lambda)^{-1}B(\lambda)\\I_m \end{bmatrix}x(\lambda)
		\end{align*}
		is a bijection between the right nullspaces of $R(\lambda)$ and $P(\lambda)$.
		\item[{\rm(b)}] The map
		\begin{align*}
		T_\ell \, : \, \mathcal{N}_\ell(R) & \longrightarrow \mathcal{N}_\ell(P) \\
		y(\lambda)^T & \longmapsto  y(\lambda)^T\begin{bmatrix}C(\lambda)A(\lambda)^{-1} & I_p \end{bmatrix}
		\end{align*}
		is a bijection between the left nullspaces of $R(\lambda)$ and $P(\lambda)$.
	\end{itemize}
\end{lem}
\begin{proof}
	We only prove part (a) since part (b) can be proved in a similar way.
	
	First, we observe that the map $T_r$ is linear.		
	Second, we notice that for any vector $x(\lambda)$, we have
	\[
	\begin{bmatrix}
	A(\lambda) & B(\lambda) \\
	-C(\lambda) & D(\lambda)
	\end{bmatrix}
	\begin{bmatrix}
	-A(\lambda)^{-1}B(\lambda)x(\lambda) \\ x(\lambda)
	\end{bmatrix} =
	\begin{bmatrix}
	0 \\ R(\lambda)x(\lambda)
	\end{bmatrix},
	\] 
	which shows that $T_r$ maps vectors in the right nullspace of $R(\lambda)$ to vectors in the right nullspace of $P(\lambda)$.
	Finally, by \eqref{ranks_rel} and the rank-nullity theorem, we have 
	\begin{equation}\label{dimequal}
	\text{dim}\;\mathcal{N}_{r}(P)=\text{dim}\;\mathcal{N}_{r}(R).
	\end{equation} 
	Since the right nullspaces of $P(\lambda)$ and $R(\lambda)$ have the same dimension and the linear map $T_r$ is clearly injective, we conclude that the map $T_r$ is bijective.
\end{proof}

\begin{rem}\label{rem: from R to P}\rm
	Since the maps in Lemma \ref{Lemma: from R to P} are bijections, they preserve linear independence.
	Hence, one can recover a basis of the right (resp. left) nullspace of $R(\la)$ from a basis of the right (resp. left) nullspace of $P(\la)$, and conversely.
	For instance, from part-(a) in Lemma \ref{Lemma: from R to P}, we obtain that  if $\{x_{i}(\la)\}_{i=1}^{t}$ is a basis of $\mathcal{N}_{r}(R)$,  then $\left\{\left[\begin{smallmatrix}-A(\la)^{-1}B(\la)x_i(\la)\\x_i(\la)\end{smallmatrix}\right]\right\}_{i=1}^{t}$ is a basis of $\mathcal{N}_{r}(P)$.
	Conversely, if $\left\{\left[\begin{smallmatrix}y_i(\la)\\x_i(\la)\end{smallmatrix}\right]\right\}_{i=1}^{t}$ is a basis of $\mathcal{N}_{r}(P)$ then $\{x_{i}(\la)\}_{i=1}^{t}$ is a basis of $\mathcal{N}_{r}(R).$ 
\end{rem}

		In Section \ref{sec:recovery}, we will show how to recover  minimal bases and minimal indices of rational matrices from those of their linearizations constructed in this paper.
		Our results will need to assume some minimality conditions. 
		For other kind of results about recovery of minimal basis and minimal indices of rational matrices from polynomial system matrices we refer the reader to \cite{onminimal,VVK79,Ver81}.

\begin{rem}\rm
All the minimal bases appearing in this work are arranged as the columns or rows of matrix polynomials.
With a slight abuse of notation, we say that an $m\times n$ matrix polynomial with $m>n$ (resp. $m < n$) is a minimal basis if its columns (resp. rows) form a minimal basis of the rational subspace they span.
\end{rem}

The following definitions are useful for characterizing minimal bases. 
\begin{deff}
The $i$th column (resp. row) degree of a matrix polynomial $B(\lambda)$ is the degree of the $i$th column (resp. row) of $B(\lambda)$.
\end{deff}
\begin{deff}
Let $B(\lambda)\in\mathbb{F}[\lambda]^{m\times n}$ be a matrix polynomial with column (resp. row) degrees
$d_1, d_2,\hdots, d_n$ (resp. $d_1, d_2,\hdots, d_n$).
The highest column (resp. row) degree coefficient matrix of $B(\lambda)$, denoted by $B_{\rm hcd}$ (resp. $B_{\rm hrd}$), is the $m\times n$ constant matrix whose $j$th column (resp. row) is the coefficient of $\lambda^{d_j}$ in the $j$th column (resp. row) of $B(\lambda)$.
 The matrix polynomial $B(\lambda)$ is called column (resp. row) reduced if $B_{\rm hcd}$ (resp. $B_{\rm hrd}$) has full colum  (resp. row) rank.
\end{deff}

 Theorem \ref{Thm:charac min bases} states a characterization of minimal bases that will be very useful in the sequel (see \cite[Main Theorem]{forney} or \cite[Theorem 2.2]{BKL}). 	
\begin{theo}\label{Thm:charac min bases}
	The columns (resp. rows) of a polynomial matrix $B(\lambda)\in\F[\la]^{m\times n}$ with $m>n$ (resp. $m < n$) are a minimal basis of the subspace they span if and only if $B(\lambda)$ is column (resp. row) reduced and $B(\lambda_0)$ has full column (resp. row) rank for all $\lambda_0\in\mathbb{F}$.
\end{theo}

We, finally, introduce the notion of dual minimal bases \cite[Definition 2.5]{BKL}.
\begin{deff}\label{dual}  Let  $B(\la)\in\F[\la]^{p\times n}$ be a minimal basis with $p<n$. Another minimal basis $N(\la)\in\F[\la]^{q\times n}$ is said to be dual to $B(\la)$ if $p+q=n$ and $B(\la)N(\la)^{T}=0$. Then, we say that $N(\la)$ is a dual minimal basis of $B(\la)$ and vice versa.
\end{deff}

\subsection{Block minimal basis linearizations of matrix polynomials}

We begin this section by reviewing the classical definitions of linearization and strong linearization of matrix polynomials \cite{Lancaster}. 

Let $P(\lambda)\in\efe[\lambda]^{p\times m}$ be a polynomial matrix as in \eqref{eq:matrix poly}.
A linear polynomial matrix (also called a matrix pencil) $L(\lambda)= B \lambda + A$ is a linearization of $P(\lambda)$ if there exist unimodular matrices $U(\lambda)$ and $V(\lambda)$ such that
\[
U(\lambda)( B \lambda + A)V(\lambda) = 
\begin{bmatrix}
	P(\lambda) & 0 \\
	0 & I_s
\end{bmatrix},
\]
for some identity matrix $I_s$.
Furthermore, a linearization $L(\lambda)$ is strong if $ A \lambda + B$ is a linearization of the polynomial matrix
\begin{equation}\label{eq:reversal poly}
\rev_d P(\lambda) :=  P_0 \lambda^d + \cdots +  P_{d-1} \lambda + P_d,
\end{equation}
where $d$ is the grade of $P(\lambda)$.

We recall that a strong linearization $L(\lambda)$ of  $P(\lambda)$ preserves the finite and infinite elementary divisors, and the dimensions of the right and left nullspaces.

The polynomial matrix \eqref{eq:reversal poly} is called the reversal of $P(\lambda)$.

We now recall the notion of (strong) block minimal basis pencil. It will be our main tool for building linearizations of rational matrices in the sense of \cite{local}.
\begin{deff} {\rm \cite[Definition 3.1]{BKL}} \emph{((Strong) block minimal basis pencil)}. \label{def:minlinearizations} A block minimal basis pencil is a linear polynomial matrix over $\mathbb{F}$ with the following structure
	\begin{equation}
	\label{eq:minbaspencil}
	L(\la) =
	\left[
	\begin{array}{cc}
	M(\la) & K_2 (\la)^T \\
	K_1 (\lambda) &0
	\end{array}
	\right],
	\end{equation}
	where $K_1(\la)$ and $K_2(\la) $ are both minimal bases. 
Moreover, given a polynomial matrix $P(\la),$ it is said that $L(\la)$ is associated with $P(\la)$ if $$N_{2}(\la) M(\la) N_{1}(\la)^{T}=P(\la),$$ where $N_{1}(\la)$ and $N_{2}(\la)$ are minimal bases dual to $K_{1}(\la)$ and $K_{2}(\la),$ respectively. 
In addition, if $K_1(\la) $ (resp. $K_2(\la)$) is a minimal basis with all its row degrees equal to $1$ and with the row degrees of a minimal basis $N_1(\la)$ (resp. $N_2(\la) $) dual to $K_1(\la)$ (resp. $K_2(\la)$) all equal, then $L(\la)$ is called a strong block minimal basis pencil. 
\end{deff}
\begin{theo}{\rm \cite{BKL}}\label{thm:BMBP are linearizations}
A  minimal basis pencil $L(\lambda)$ as in \eqref{eq:minbaspencil} associated with a polynomial matrix $P(\lambda)$ is a  linearization of $P(\lambda)$. Moreover, if the block minimal basis $L(\lambda)$ is strong, then $L(\lambda)$ is a strong linearization of $P(\lambda)$ considered as a polynomial matrix of grade $\deg N_1(\lambda) + \deg N_2(\lambda) + 1$.
\end{theo}

\begin{rem}\rm
In this paper, we only consider the so-called degenerate (strong) block minimal basis pencils, that is, (strong) block minimal basis pencils of the form 
\begin{equation}
\label{eq:minbaspencilonecolumn}
L(\la) =
\left[
\begin{array}{c}
M(\la)  \\
K (\lambda) 
\end{array}
\right],
\end{equation}
where $K(\lambda)$ is a minimal basis. The linear polynomial matrix \eqref{eq:minbaspencilonecolumn} is a (strong) linearization of the polynomial matrix
\[
P(\lambda) = M(\lambda) N(\lambda)^T
\]
considered as a polynomial of grade $1+\deg N(\lambda)$, where $N(\lambda)$ is a minimal basis dual to $K(\lambda)$.
If $d:=\deg N(\lambda)+1$ and $P(\lambda)$ is of size $m\times n$, then $M(\lambda)$ and $K(\lambda)$ in \eqref{eq:minbaspencilonecolumn} are, respectively,  $m\times dn$ and $(d-1)n\times dn$ linear polynomial matrices \cite{BKL}.
\end{rem}

\section{Linearizations in a set $\Omega\subseteq\mathbb{F}$}\label{sec:local}

In Definition \ref{def:local lin}, we recall the notion of local linearization of a rational matrix which has been recently introduced in \cite{local}.
\begin{deff}[Local  linearization of rational matrices] {\rm \cite{local}}
\label{def:local lin}
Let $R(\lambda)\in\mathbb{F}(\lambda)^{p\times m}$ be a rational matrix.
Let
\[
\mathcal{L}(\lambda) = 
\left[\begin{array}{c|c}
A_1 \la +A_0 &B_1 \la +B_0\\\hline \phantom{\Big|} -(C_1 \la +C_0)&D_1 \la +D_0
\end{array}\right]\in\mathbb{F}[\lambda]^{(n+(p+s))\times(n+(m+s))}
\] 
be a linear polynomial system matrix with transfer function matrix
\[
\widehat{R}(\lambda) = D_1\lambda+D_0+(C_1\lambda+C_0)(A_1\lambda+A_0)^{-1}(B_1\lambda+B_0) \in \mathbb{F}[\lambda]^{(p+s)\times (m+s)}.
\]
Then, $\mathcal{L}(\lambda)$ is a linearization of $R(\lambda)$ in a set $\Omega\subseteq\mathbb{F}$ if the following conditions hold:
\begin{enumerate}
\item[\rm(a)] $\mathcal{L}(\lambda)$ is minimal in $\Omega$, and
\item[\rm(b)] there exist rational matrices $R_1(\lambda)$ and $R_2(\lambda)$ both regular in $\Omega$ such that
\[
R_1(\lambda)\begin{bmatrix}
R(\lambda) & 0 \\
0 & I_s
\end{bmatrix} R_2(\lambda) = \widehat{R}(\lambda).
\]
\end{enumerate}
\end{deff}
Theorem \ref{thm:spectral charac} presents the spectral characterization of local linearizations.
In words, this theorem says that the local spectral information (in terms of zero and pole elementary divisors) of $R(\lambda)$ in $\Omega$ can be recovered from the local spectral information (in terms of elementary divisors) of any of its linearizations $\mathcal{L}(\lambda)$ in $\Omega$.
\begin{theo}[Spectral characterization of local linearizations]{\rm \cite{local}}
\label{thm:spectral charac}
Let $R(\lambda)\in\mathbb{F}(\lambda)^{p\times m}$ be a rational matrix.
Let
\[
\mathcal{L}(\lambda) = 
\left[\begin{array}{c|c}
A_1 \la +A_0 &B_1 \la +B_0\\\hline \phantom{\Big|} -(C_1 \la +C_0)&D_1 \la +D_0
\end{array}\right]\in\mathbb{F}[\lambda]^{(n+(p+s))\times(n+(m+s))}
\] 
be a linear polynomial system matrix, with state matrix $A_1\la + A_0$, minimal in $\Omega\subseteq\F$.
Then, $\mathcal{L}(\lambda)$ is a linearization of $R(\lambda)$ in $\Omega$ if and only if
\begin{enumerate}
\item[\rm(a)] $\rank\mathcal{L}(\lambda) = \rank R(\lambda)+n+s$,
\item[\rm(b)] the pole elementary divisors of $R(\lambda)$ in $\Omega$ are the elementary divisors of $A_1\lambda+A_0$ in $\Omega$, and
\item[\rm(c)] the zero elementary divisors of $R(\lambda)$ in $\Omega$ are the elementary divisors of $\mathcal{L}(\lambda)$ in $\Omega$.
\end{enumerate}
\end{theo}

In Theorem \ref{thm:linearization}, we construct (local) linearizations for rational matrices that are represented with realizations as in \eqref{eq_realizR}.
To prove Theorem \ref{thm:linearization}, we will use Lemma \ref{lemma:unimodular}, whose simple proof is omitted.

\begin{lem}\label{lemma:unimodular}
	A polynomial matrix of the form
	\[
	\begin{bmatrix}
	X(\lambda) & A(\lambda) & Y(\lambda) & B(\lambda) \\
	I_{s} & 0 & 0 & 0 \\
	Z(\lambda) & -C(\lambda) & W(\lambda) & D(\lambda) \\
	0 & 0 & I_{t} & 0 
	\end{bmatrix}
	\]
	is unimodularly equivalent to 
	$\diag\left( 
	\left[\begin{smallmatrix}
	A(\lambda) & B(\lambda) \\
	C(\lambda) & D(\lambda)
	\end{smallmatrix}\right], I_{s+t}
	\right)$.
\end{lem}

\begin{theo}\label{thm:linearization} 
Let $R(\la)\in\F(\la)^{p\times m}$ be a rational matrix expressed in the form $R(\la)=D(\la)+C(\la)A(\la)^{-1}B(\la)$, for some nonsingular polynomial matrix  $A(\la)\in\FF[\la]^ {n\times n}$, and polynomial matrices
	$B(\la)\in\FF[\la]^{n\times m}$, $C(\la)\in\FF[\la]^{p\times n}$ and $D(\la)\in\FF[\la]^{p\times m}$. 
	Let 
	\begin{equation}
	L_{A}(\la) =
	\left[
	\begin{array}{c}
	M_{A}(\la)  \\
	K_{A} (\lambda) 
	\end{array}
	\right] \quad \text{ and } \quad L_{D}(\la) =
	\left[
	\begin{array}{c}
	M_{D}(\la)  \\
	K_{D} (\lambda) 
	\end{array}
	\right]
	\end{equation}
	 be block minimal basis linearizations of $A(\la)$ and $D(\la),$ respectively.
	 Let $N_{A}(\la)$ and $N_{D}(\la)$ be minimal bases dual to $K_{A}(\la)$ and $K_{D}(\la),$ respectively. 
	 Consider linear polynomial matrices $M_{C}(\la)$ and $M_{B}(\la)$ such that 
	\begin{equation}
	M_{C}(\la)N_{A}(\la)^{T}=C(\la)\quad \text{ and }\quad\ M_{B}(\la)N_{D}(\la)^{T}=B(\la), 
	\end{equation} 
	and the linear polynomial system matrix  
	\begin{equation} \label{eq:linearization}
	\mathcal{L}(\la) = \left[
	\begin{array}{c|c}
	M_A(\la) & \phantom{a} M_B(\la) \phantom{a} \\
	K_A(\la) & 0 \\ \hline \phantom{\Big|}
	-M_C(\la)  \phantom{\Big|}& M_D(\la) \\
	0 & K_D (\la)
	\end{array}
	\right],
	\end{equation} 
with state matrix $L_{A}(\la).$ If the matrices
\begin{equation} \label{eq:mat_minimality}
	\begin{bmatrix} A(\la) \\ C(\la)\end{bmatrix}\quad \text{and} \quad 	\begin{bmatrix} A(\la) & B(\la) \end{bmatrix}
\end{equation}
have no eigenvalues in $\Omega\subseteq\mathbb{F}$, then $\mathcal{L}(\la) $ is a linearization of $R(\la)$ in $\Omega$.
\end{theo}
\begin{rem}\label{rem_unimodular}\rm Before giving the proof of Theorem \ref{thm:linearization}, we recall that there exist unimodular polynomial matrices of the form
	\begin{equation}\label{eq:embedding}
	U_i(\lambda) =
	\begin{bmatrix}
	K_i(\lambda) \\ \widehat{K}_i(\la)
	\end{bmatrix}
	\quad \mbox{and} \quad
	U_i(\lambda)^{-1}=
	\begin{bmatrix}
	\widehat{N}_i(\lambda)^T & N_i(\lambda)^T
	\end{bmatrix},
	\end{equation}
	for $i\in\{A,D\}$; see \cite[Theorem 2.10]{BKL}.
\end{rem}
\begin{proof}
Throughout the proof, we use the notation $\rho_A:=\deg N_{A}(\la)$ and $\rho_D:=\deg N_{D}(\la)$.



To prove that $\mathcal{L}(\lambda)$ is a linearization of $R(\lambda)$ in $\Omega$, we will use the characterization of local linearizations in Theorem \ref{thm:spectral charac}.

Let $P(\lambda):=\left[ \begin{smallmatrix} A(\lambda) & B(\lambda) \\ C(\lambda) & D(\lambda) \end{smallmatrix} \right]$ be a polynomial system matrix of $R(\lambda)$.
Notice that $P(\lambda)$ is minimal in $\Omega$ by hypothesis.
First, we have
\begin{alignat*}{2}
\mbox{pole elem. div. of $R(\lambda)$ in $\Omega$} =& \mbox{elem. div. of $A(\lambda)$ in $\Omega$} \quad  &&\mbox{($P$ is minimal in $\Omega$)}\\
=&\mbox{elem. div. of $L_A(\lambda)$ in $\Omega$} \quad &&\mbox{($L_A$ is a linearization of $A$)}.
\end{alignat*}
Hence, the pole elementary divisors of $R(\lambda)$ in $\Omega$ are equal to the elementary divisors of $L_A(\lambda)$ in $\Omega$. 

Second, we consider Remark \ref{rem_unimodular} and notice $K_A(\lambda)\widehat{N}_A(\lambda)^T = nI_{\rho_A}$ and $K_D(\lambda)\widehat{N}_D(\lambda)^T = mI_{\rho_D}$, as this will be important in what follows.
Then, multiplying $\mathcal{L}(\lambda)$ on the left by the unimodular matrix $U(\la)=\diag(U_A(\la)^{-1},U_{D}(\la)^{-1})$, we get
\begin{equation}\label{eq:unimodular1}
\mathcal{L}(\la)U(\la)=\left[
\begin{array}{cc|cc}
X_{AA}(\la) & \phantom{a} A(\la) \phantom{a} &X_{BD}(\la) & B(\la) \\
I_{n\,\rho_A} & 0 & 0 & 0 \\ \hline \phantom{\Big|}
-X_{CA}(\la) & -C(\la) & X_{DD}(\la) & D(\la)  \\
0 & 0 & I_{m\,\rho_D} & 0
\end{array}
\right] ,
\end{equation}
 where 
\begin{equation}\label{eq:exis}
X_{ij}(\la):=M_{i}(\la)\wh N_{j}(\la)^{T}, \quad \quad  i,j\in\{A,D\}.
\end{equation}
Since the matrices in \eqref{eq:mat_minimality} have no eigenvalues in $\Omega$, we have that $\mathcal{L}(\la)U(\la)$ is minimal in $\Omega$ and, therefore, $\mathcal{L}(\la)$ is minimal in $\Omega$. By Lemma \ref{lemma:unimodular},  $\mathcal{L}(\lambda)U(\lambda)$ is unimodularly equivalent to $\diag(P(\lambda),I_{n\,\rho_A+m\,\rho_B})$.
Hence, $\mathcal{L}(\lambda)$ is a linearization of the polynomial system matrix $P(\lambda)$.
As a consequence of this, we have
\begin{alignat*}{2}
\mbox{zero elem. div. of $R(\lambda)$ in $\Omega$} =& \mbox{elem. div. of $P(\lambda)$ in $\Omega$} \quad \quad &&\mbox{($P$ is minimal in $\Omega$)}\\
=&\mbox{elem. div. of $\mathcal{L}(\lambda)$ in $\Omega$} \quad \quad &&\mbox{($\mathcal{L}$ is a linearization of $P$)}.
\end{alignat*}
Therefore, the zero elementary divisors of $R(\lambda)$ in $\Omega$ are equal to the elementary divisors of $P(\lambda)$ in $\Omega$. 

Since $\mathcal{L}(\lambda)$ is of size $(n+p+s)\times(n+m+s)$, where $s=n\rho_A+m\rho_B$, to finish the proof, it suffices to notice that 
\begin{alignat*}{2}
\rank \mathcal{L}(\lambda) =& \rank P(\lambda) + n\rho_A+m\rho_B \quad \quad && \mbox{(by \eqref{eq:unimodular1})} \\
=& \rank R(\lambda) + n + n\rho_A+m\rho_B \quad \quad &&\mbox{(by \eqref{ranks_rel})}.
\end{alignat*}
By Theorem \ref{thm:spectral charac}, we conclude that $\mathcal{L}(\lambda)$ is a linearization of $R(\lambda)$ in $\Omega$.
\end{proof}

In Example \ref{ex:1}, we show how to use Theorem \ref{thm:linearization} to construct linearizations of a rational matrix of the form \eqref{eq_realizR}.
For simplicity, we assume that the polynomial matrices $A(\lambda)$, $B(\lambda)$, $C(\lambda)$ and $D(\lambda)$ are expressed in the monomial basis.
But we emphasize that the construction can be easily adapted to many other polynomial bases (Chebyshev, Lagrange, Newton, etc).

\begin{example}\label{ex:1}
Let us consider a rational matrix of the form
\begin{align*}
R(\lambda) =& D(\lambda) + C(\lambda)A(\lambda)^{-1} B(\lambda) = \\ & D_3\lambda^3+D_2\lambda^2+ D_1\lambda + D_0  + \\&
\hspace{2cm} ( C_1\lambda+C_0)(A_3\lambda^3+A_2\lambda^2+ A_1\lambda + A_0)^{-1}(B_2\lambda^2 +  B_1\lambda + B_0),
\end{align*}
where $A(\la)\in\FF[\la]^ {n\times n}$ is regular, and $B(\la)\in\FF[\la]^{n\times m}$, $C(\la)\in\FF[\la]^{p\times n}$ and $D(\la)\in\FF[\la]^{p\times m}$.
We will use block Kronecker pencils \cite[Section 4]{BKL}, which are particular cases of block minimal basis pencils, to construct the linearizations $L_A(\lambda)$ and $L_D(\lambda)$ in Theorem \ref{thm:linearization}. We recall that the construction of block Kronecker pencils involves a pair of dual minimal bases of the form
\[
K(\lambda) = 
\begin{bmatrix}
	-I_s & I_s \lambda & 0 & \cdots & 0 \\
	0 & -I_s & I_s \lambda & \ddots & \vdots \\
	\vdots & \ddots & \ddots & \ddots & 0 \\
	0 & \cdots & 0 & -I_s & I_s \lambda
\end{bmatrix}
\quad \mbox{and} \quad
N(\lambda)^T =
\begin{bmatrix}
	I_s\lambda^{d-1} \\ \vdots \\ I_s \lambda \\ I_s
\end{bmatrix}.
\]

To construct the linearization $L_A(\lambda)$ and the linear polynomial matrix $M_C(\lambda)$, we need to see both $A(\lambda)$ and $C(\lambda)$ as polynomial matrices of grade
\[
\max\{\mathrm{deg}\, A(\lambda), \mathrm{deg} \, C(\lambda) \} = \max\{3,1\} = 3.
\]
Then, we can use, for example, 
\[
L_A(\lambda) := 
\begin{bmatrix}
	 A_3 \lambda + A_2 & A_1 & A_0 \\
	-I_n &  I_n \lambda& 0 \\
	0 & -I_n &  I_n \lambda
\end{bmatrix}
\quad \mbox{and} \quad 
M_C(\lambda) := 
\begin{bmatrix}
	0 & C_1 & C_0
\end{bmatrix},
\]
with $N_A(\la):=\begin{bmatrix}
 I_n\la^2 &  I_n\la & I_n
\end{bmatrix}$. Similarly, to construct the linearization $L_D(\lambda)$ and the linear polynomial matrix $M_B(\lambda)$, we need to see both $D(\lambda)$ and $B(\lambda)$ as polynomial matrices of grade
\[
\max\{\mathrm{deg}\, D(\lambda), \mathrm{deg}\, B(\lambda) \} = \max\{3,2\} = 3.
\]
Then, we can use, for instance,
\[
L_D(\lambda) := 
\begin{bmatrix}
	D_3\la +D_2 & D_1 & D_0\\
	-I_m &  I_m\lambda & 0 \\
	0 & -I_m &  I_m\lambda
\end{bmatrix}
\quad \mbox{and} \quad 
M_B(\lambda) := 
\begin{bmatrix}
	 0 & B_2\la +B_1 & B_0
\end{bmatrix},
\]
with $N_D(\la):=\begin{bmatrix}
I_m\la^2 &  I_m\la & I_m
\end{bmatrix}$. Then the linear polynomial system matrix $\mathcal{L}(\lambda)$ in Theorem \ref{thm:linearization} is
\[
\mathcal{L}(\la) = 
\left[ \begin{array}{ccc|ccc}
 A_3\lambda+A_2 & A_1 & A_0 &  0 & B_2\la +B_1 & B_0 \\
-I_n &  I_n\lambda & 0 & 0 & 0 & 0 \\
0 & -I_n &I_n \lambda  & 0 & 0 & 0 \\ \hline
0 & -C_1 & - C_0 & D_3\la +D_2 & D_1 & D_0 \\
0 & 0 & 0 & -I_m &  I_m \lambda& 0 \\
0 & 0 & 0 & 0 & -I_m &  I_m \lambda
\end{array}\right].
\]
Theorem \ref{thm:linearization} guarantees that $\mathcal{L}(\lambda)$ is a linearization of $R(\lambda)$ wherever the polynomial matrices $	\begin{bmatrix} A(\la) \\ C(\la)\end{bmatrix}$ and $\begin{bmatrix} A(\la) & B(\la) \end{bmatrix}$ do not have eigenvalues.
\end{example}

\section{Linearizations at infinity}\label{sec:infinity}

In this section, we study when the pole and zero eigenstructure at infinity of a rational matrix $R(\lambda)$  can be recovered from that of any of its linearizations in Theorem \ref{thm:linearization}.

We begin by recalling the definition of $g$-reversal \cite{local}.
Given an integer $g$ and a rational matrix $R(\lambda)$, the $g$-reversal of the rational matrix $R(\lambda)$ is the rational matrix 
\[
\rev_g R(\lambda) := \lambda^g R\left( \frac{1}{\lambda} \right).
\]
If $R(\lambda)$ is a polynomial matrix and $g=\deg R(\lambda)$, then $\rev_g R(\lambda)$ is a matrix polynomial, which we denote by $\rev R(\lambda)$. 
An expression such as $\rev_g R(\lambda_0)$ should be understood as $\rev_g R(\lambda)$ evaluated at $\lambda=\lambda_0$.

In Definition \ref{def:lin at infinity}, we recall the notion of linearization at infinity of grade $g$ which has been recently introduced in \cite{local}.
\begin{deff}[Linearization at infinity of grade $g$]{\rm \cite{local}}
\label{def:lin at infinity}
Let $R(\lambda)\in\mathbb{F}(\lambda)^{p\times m}$ be a rational matrix and let
\[
\mathcal{L}(\lambda) = 
\left[\begin{array}{c|c}
A_1 \la +A_0 &B_1 \la +B_0\\\hline \phantom{\Big|} -(C_1 \la +C_0)&D_1 \la +D_0
\end{array}\right]\in\mathbb{F}[\lambda]^{(n+(p+s))\times(n+(m+s))}
\] 
be a nonconstant linear polynomial system matrix with transfer function matrix
\[
\widehat{R}(\lambda) = D_1\lambda+D_0+(C_1\lambda+C_0)(A_1\lambda+A_0)^{-1}(B_1\lambda+B_0) \in \mathbb{F}[\lambda]^{(p+s)\times (m+s)}.
\]
Let $g$ be an integer.
Then, $\mathcal{L}(\lambda)$ is a linearization of $R(\lambda)$ at infinity of grade $g$ if the following conditions hold:
\begin{enumerate}
\item[\rm(a)] $\rev \mathcal{L}(\lambda)$ is minimal at 0, and
\item[\rm(b)] there exist rational matrices $Q_1(\lambda)$ and $Q_2(\lambda)$ both regular at $0$ such that
\[
Q_1(\lambda)\begin{bmatrix}
\rev_g R(\lambda) & 0 \\
0 & I_s
\end{bmatrix} Q_2(\lambda) = \rev_1 \widehat{R}(\lambda).
\]

\end{enumerate}
\end{deff}

Theorem \ref{thm:spectral char at infinity} characterizes linearizations  at infinity of grade $g$ in terms of their spectral information.
\begin{theo}[Spectral characterization of linearizations at infinity  of grade $g$]{\rm \cite{local}}\label{thm:spectral char at infinity}
Let $R(\lambda)\in\mathbb{F}(\lambda)^{p\times m}$ be a rational matrix and let
\[
\mathcal{L}(\lambda) = 
\left[\begin{array}{c|c}
A_1 \la +A_0 &B_1 \la +B_0\\\hline \phantom{\Big|} -(C_1 \la +C_0)&D_1 \la +D_0
\end{array}\right]\in\mathbb{F}[\lambda]^{(n+q)\times(n+r)}
\] 
be a nonconstant linear polynomial system matrix, with state matrix $A_1\la+A_0$, such that $\rev \mathcal{L}(\lambda)$ is minimal at 0. Let $g$ be an integer.
Then $\mathcal{L}(\lambda)$ is a linearization of $R(\lambda)$ at infinity of grade $g$ if and only if 
\begin{enumerate}
\item[\rm (a)] $\rank \mathcal{L}(\lambda) = R(\lambda)+n+s$,
\item[\rm (b)] the pole elementary divisors of $\rev_g R(\lambda)$ at 0 are the elementary divisors of $\rev_1 A_1\lambda + A_0$ at 0, and
\item[\rm (c)] the zero elementary divisors of $\rev_g R(\lambda)$ at 0 are the elementary divisors of $\rev \mathcal{L}(\lambda)$ at 0.
\end{enumerate}
\end{theo}

An important property of linearizations at infinity of grade $g$ of a rational matrix $R(\lambda)$ is that they allow for the recovery of the complete spectral information of $R(\lambda)$ at infinity, as Theorem \ref{thm:recovery at infinity} shows.
\begin{theo}[Recovery of the spectral information at infinity]{\rm \cite{local}}
\label{thm:recovery at infinity}
Let $R(\lambda)\in\mathbb{F}(\lambda)^{p\times m}$ be a rational matrix and let
\[
\mathcal{L}(\lambda) = 
\left[\begin{array}{c|c}
A_1 \la +A_0 &B_1 \la +B_0\\\hline \phantom{\Big|} -(C_1 \la +C_0)&D_1 \la +D_0
\end{array}\right]\in\mathbb{F}[\lambda]^{(n+q)\times(n+r)}
\] 
be a linearization at infinity of grade $g$ of $R(\lambda)$, with $ \mathcal{L}(\lambda)$ nonconstant. 
If $e_1\leq \cdots \leq e_t$ are the partial multiplicities of $\rev_1 (A_1\lambda+A_0)$ at 0 and $\widetilde{e}_1\leq \cdots \leq \widetilde{e}_u$ are the partial multiplicities of $\rev \mathcal{L}(\lambda)$ at 0, then the invariant orders at infinity $q_1\leq \cdots \leq q_r$ of $R(\lambda)$ are
\[
(q_1,\hdots,q_r) = (-e_t,\hdots,-e_1,\underbrace{0,\hdots,0}_{r-t-u},\widetilde{e}_1,\hdots,\widetilde{e}_u)-(g,\hdots,g).
\] 
\end{theo}

Theorem \ref{thm:linearization at infinity} is the main result of this section.
It shows that, under some mild conditions, the local linearizations introduced in Section \ref{sec:local} are also linearizations at infinity of the rational matrix $R(\lambda)$.
\begin{theo}\label{thm:linearization at infinity} 
Let $R(\la)\in\F(\la)^{p\times m}$ be a rational matrix expressed in the form $R(\la)=D(\la)+C(\la)A(\la)^{-1}B(\la)$, for some nonsingular polynomial matrix  $A(\la)\in\FF[\la]^ {n\times n}$, and polynomial matrices
	$B(\la)\in\FF[\la]^{n\times m}$, $C(\la)\in\FF[\la]^{p\times n}$ and $D(\la)\in\FF[\la]^{p\times m}$. 
	Let 
	\begin{equation*}
	L_{A}(\la) =
	\left[
	\begin{array}{c}
	M_{A}(\la)  \\
	K_{A} (\lambda) 
	\end{array}
	\right] \quad \text{ and } \quad L_{D}(\la) =
	\left[
	\begin{array}{c}
	M_{D}(\la)  \\
	K_{D} (\lambda) 
	\end{array}
	\right]
	\end{equation*}
	 be strong block minimal basis linearizations of $A(\la)$ and $D(\la),$ respectively.
	 Let $N_{A}(\la)$ and $N_{D}(\la)$ be minimal bases dual to $K_{A}(\la)$ and $K_{D}(\la),$ respectively, and denote $\rho_A:=\deg N_{A}(\la)$, $\rho_D:=\deg N_{D}(\la)$. 
	 Consider linear polynomial matrices $M_{C}(\la)$ and $M_{B}(\la)$ such that 
	\begin{equation*}
	M_{C}(\la)N_{A}(\la)^{T}=C(\la)\quad \text{ and }\quad\ M_{B}(\la)N_{D}(\la)^{T}=B(\la), 
	\end{equation*} 
	and the linear polynomial system matrix  
	\begin{equation} \label{eq:linearization}
	\mathcal{L}(\la) = \left[
	\begin{array}{c|c}
	M_A(\la) & \phantom{a} M_B(\la) \phantom{a} \\
	K_A(\la) & 0 \\ \hline \phantom{\Big|}
	-M_C(\la)  \phantom{\Big|}& M_D(\la) \\
	0 & K_D (\la)
	\end{array}
	\right],
	\end{equation} 
with state matrix $L_{A}(\la).$ 
If the matrices
	\begin{equation}\label{mat_minimality_inf}
	\begin{bmatrix} \rev_{\rho_A+1} A(\la) \\ \rev_{\rho_A +1} C(\la)\end{bmatrix}\quad \text{and} \quad \begin{bmatrix} \rev_{\rho_A +1} A(\la) & \rev_{\rho_D+1} B(\la)\end{bmatrix}
	\end{equation}
	have no eigenvalues at $0$, then $\mathcal{L}(\la)$ is a linearization of $R(\la)$ at infinity of grade $\rho_D+1$.
\end{theo}
\begin{rem}\rm
Since $L_A(\lambda)$ and $L_B(\lambda)$ are strong block minimal basis pencils, we recall that there exist unimodular polynomial matrices of the form
\begin{equation}\label{eq:embedding2}
\widetilde{U}_i(\lambda) =
\begin{bmatrix}
\rev_{1}K_{i}(\la)\\ \widetilde{K}_i(\la)
\end{bmatrix}
\quad \mbox{and} \quad
\widetilde{U}_i(\lambda)^{-1}=
\begin{bmatrix}
\widetilde{N}_i(\lambda)^T & \rev_{\rho_i} N_i(\lambda)^T
\end{bmatrix},
\end{equation}
for $i\in\{A,B\}$; see \cite[Theorem 2.10]{BKL}.
\end{rem}
\begin{proof}
To prove that $\mathcal{L}(\lambda)$ is a linearization at infinity of grade $g$ of $R(\lambda)$, we will use the characterization in Theorem \ref{thm:spectral char at infinity}.

Let $g:=\rho_D + 1$.
Let us consider the polynomial system matrix
\begin{equation}\label{eq:poly system theorem}
\widetilde{P}(\lambda):=\begin{bmatrix}
\rev_{\rho_A+1}A(\lambda) & \rev_{\rho_D+1} B(\lambda) \\
\rev_{\rho_A+1} C(\lambda) & \rev_{\rho_D+1} D(\lambda)
\end{bmatrix},
\end{equation}
with state space matrix $\rev_{\rho_A+1}A(\lambda)$.
We observe that the transfer function of the polynomial system matrix \eqref{eq:poly system theorem} is
\[
\rev_{\rho_D+1}D(\lambda) + \rev_{\rho_A+1}C(\lambda)\left(\rev_{\rho_A+1}A(\lambda) \right)^{-1}\rev_{\rho_D+1}B(\lambda) = \rev_g R(\lambda).
\]
Hence, we have
\begin{alignat*}{2}
\mbox{pole elem. div. of $\rev_g R(\lambda)$ at 0} =& \mbox{elem. div. of $\rev_{\rho_A+1}A(\lambda)$ at 0} \quad  &&\mbox{($\widetilde{P}$ is minimal at 0)}\\
=&\mbox{elem. div. of $\rev_1 L_A(\lambda)$ in $\Omega$} \quad &&\mbox{($L_A$ is  strong)}.
\end{alignat*}
Thus, the pole elementary divisors of $\rev_g R(\lambda)$ at 0 are equal to the elementary divisors of $\rev_1 L_A(\lambda)$ at 0.

Multiplying $\rev_1 \mathcal{L}(\lambda)$ on the left by the unimodular matrix $\widetilde{U}(\lambda) = \diag(\widetilde{U}_A(\lambda)^{-1},\widetilde{U}_D(\lambda)^{-1})$, where $\widetilde{U}_A(\lambda)^{-1}$ and $\widetilde{U}_D(\lambda)^{-1}$ are defined in \eqref{eq:embedding2}, we get 
\begin{equation}\label{eq:unimodular3}
\rev_1\mathcal{L}(\la) \, \widetilde{U}(\lambda) =
\left[
\begin{array}{cc|cc}
* & \phantom{a} \rev_{\rho_A+1} A(\la) \phantom{a} & * & \rev_{\rho_D+1} B(\la) \\
I_{n\,\rho_A} & 0 & 0 & 0 \\ \hline \phantom{\Big|}
* & -\rev_{\rho_A+1} C(\la) & * & \rev_{\rho_D+1} D(\la)  \\
0 & 0 & I_{m\rho_D} & 0
\end{array}
\right],
\end{equation}
where $*$ denotes polynomial matrices that are not important for the argument. Since the matrices in \eqref{mat_minimality_inf} have no eigenvalues at $0$, we have that $\rev_1\mathcal{L}(\la) \, \widetilde{U}(\lambda)$ is minimal at $0$ and, therefore, $\rev_1\mathcal{L}(\la)$ is minimal at $0$.
By Lemma \ref{lemma:unimodular}, the matrix polynomial $\rev_1\mathcal{L}(\la) \, \widetilde{U}(\lambda)$ is unimodularly equivalent to $\diag(\widetilde{P}(\lambda), I_{n\,\rho_A + m\, \rho_B})$.
Hence, $\rev_1 \mathcal{L}(\lambda)$ is a linearization of $\widetilde{P}(\lambda)$. 
Thus, we have 
\begin{alignat*}{2}
\mbox{zero elem. div. of $\rev_g R(\lambda)$ at 0} =& \mbox{elem. div. of $\widetilde{P}(\lambda)$ at 0} \quad &&\mbox{($\widetilde{P}$ is minimal at 0)}\\
=&\mbox{elem. div. of $\rev_1 \mathcal{L}(\lambda)$ at 0} \quad  &&\mbox{($\rev_1 \mathcal{L}$ is a linearization of $\widetilde{P}$)}.
\end{alignat*}
Therefore, the zero elementary divisors of $\rev_g R(\lambda)$ at 0 are equal to the elementary divisors of $\rev_1 \mathcal{L}(\lambda)$ at 0.

To finish the proof, it suffices to notice that 
\begin{alignat*}{2}
\rank \rev_1 \mathcal{L}(\lambda) =& \rank \widetilde{P}(\lambda) + n\rho_A+m\rho_B \quad \quad && \mbox{(by \eqref{eq:unimodular3})} \\
=& \rank \rev_g R(\lambda) + n + n\rho_A+m\rho_B \quad \quad &&\mbox{(by \eqref{ranks_rel})}.
\end{alignat*}
Conclusively, by Theorem \ref{thm:spectral char at infinity}, the linear polynomial matrix system $\mathcal{L}(\lambda)$ is a linearization of $R(\lambda)$ at infinity of grade $g=\rho_D+1$.
\end{proof}

%

\begin{rem}\rm
The linearization $ \mathcal{L}(\la)$ in Theorem \ref{thm:linearization} contains the spectral information of the rational matrix $R(\la)$ in a particular set $\Omega$ whenever certain minimality conditions are satisfied. More precisely, if the matrix polynomials in \eqref{eq:mat_minimality} have no eigenvalues in $\Omega$ then the zeros in $\Omega$ of $R(\la)$ are the eigenvalues in $\Omega$ of $ \mathcal{L}(\la)$, and the poles in $\Omega$ of $R(\la)$ are the eigenvalues in $\Omega$ of the block minimal basis pencil $L_{A}(\la)$.
If $\Omega=\F$ then we can recover the complete information about finite poles and zeros of $R(\la)$ from $\mathcal{L}(\la).$
If, in addition, the matrix polynomials \eqref{mat_minimality_inf} have no eigenvalues at zero, then $ \mathcal{L}(\la)$ is a strong linearization at infinity of grade $\rho_D+1$ and, hence, by Theorem \ref{thm:recovery at infinity}, we can also recover the complete zero and pole information of $R(\la)$ at infinity. 
\end{rem}
\section{An illustrative example}\label{example}

Let us consider an $m\times m$ rational matrix of the form 
	\begin{equation}\label{ex_rationalmatrix}
	R(\la)=D(\la) + f_{1}(\la) K_{1} + f_2(\la)K_2,
	\end{equation}
	where $D(\la)$ is a polynomial matrix of degree $2$, $f_{1}(\la)=\frac{(\la^{2}+1)(\la+2)}{\la^2-\la-2}$, $f_2(\la)=\frac{(\la^2-1)\la^2}{\la+2}$, and $K_{1}$ and $K_{2}$ are constant matrices having ranks $r_1$ and $r_2$, respectively. 
For $i=1,2$, we can write $K_{i}=c_{i}b_{i}^{T}$, for some $c_i$ and $b_i$ both of size $m\times r_i$. 
Then, a realization of $R(\lambda)$ as in \eqref{eq_realizR} is
	\begin{equation*}\label{ex_rationalmatrix_realization}
	R(\la)=D(\la) + \underbrace{ \begin{bmatrix} c_1 (\la^2+1) &c_2  (\la^2-1)	\end{bmatrix}}_{:=C(\la)} \underbrace{\begin{bmatrix}
	I_{r_1}(\la^2-\la-2) & 0 \\
	0 & I_{r_2}(\la+2)
	\end{bmatrix}^{-1}}_{:=A(\la)^{-1}}\underbrace{ \begin{bmatrix} b_{1}^{T} (\la+2) \\[3pt] b_{2}^{T} \la^{2}
	\end{bmatrix}}_{:=B(\la)}.
	\end{equation*}
Notice that the polynomial matrices $A(\lambda)$, $B(\lambda)$ and $C(\lambda)$  have degree 2.
Hence, we can write $A(\la):= A_{2}\la^{2}+ A_1 \la + A_0$, $B(\la):= B_{2}\la^{2}+ B_1 \la  + B_0$, $C(\la):= C_{2}\la^{2}+ C_1 \la + C_0$ and $D(\la):= D_{2} \la^{2}+ D_1 \la + D_0$.
 Then, we set $r:=r_1+r_2$, and we define the linear polynomial system matrix
	\begin{equation*}
	\mathcal{L}(\la) = \left[
	\begin{array}{cc|cc}
 A_2 \la + A_1 & A_0 \phantom{a}&  B_2 \la + B_1 & B_0  \\
-I_{r} &  I_{r} \la & 0 & 0  \\ \hline \phantom{\Big|}
- C_2\la -C_1 & -C_0  \phantom{\Big|}& D_2 \la  + D_1 & D_0 \\
	0 & 0 & -I_m &  I_m  \la
	\end{array}
	\right].
	\end{equation*}
	It can be proved that condition \eqref{eq:mat_minimality} is satisfied for all $\la_0\in\F$ and that condition \eqref{mat_minimality_inf} is also satisfied. 
Thus, by Theorem \ref{thm:linearization}, $	\mathcal{L}(\la)$ is a linearization of $R(\la)$ in $\F$ and also a linearization of $R(\la)$ at infinity of grade $2$. 

	Notice that the rational functions $f_1(\la)$ and $f_{2}(\la)$ in \eqref{ex_rationalmatrix} are not (strictly) proper.
Nevertheless, we have been able to construct a strong linearization  in the sense of \cite{local} for $R(\la)$ without decomposing $R(\lambda)$ as the sum of its polynomial part and its strictly proper rational part and without considering a minimal state space realalization of the strictly proper part, in contrast to the methods appearing in \cite{AlBe16, strong,dopmarquin2019, SuBai}. 
We emphasize that, in order to construct linearizations of rational matrices, we must take into account that a realization as in \eqref{eq_realizR} is not unique and that the ideal thing is to consider a realization easy to build from the original expression of the corresponding rational matrix without performing many computations.

%

\section{Recovery of eigenvectors, minimal bases and minimal indices} \label{sec:recovery}

In this section, we show how to recover the elements in the right and left nullspaces of a rational matrix $R(\la)$ (regular or singular) from those in the right and left nullspaces of a linearization $\mathcal{L}(\la)$ as in Theorem \ref{thm:linearization}, as well as minimal bases, minimal indices and eigenvectors.
To do this, we consider the transfer function matrix $\wh R(\la)$ of $\mathcal{L}(\la)$ and a polynomial system matrix $P(\la)$ of $R(\la)$, and we study the relation between the elements in their right and left nullspaces.
We will follow the scheme:
\begin{equation*}\label{scheme}
\mathcal{L}(\la) \xRightarrow[\text{matrix (Lemma \ref{Lemma: from R to P}) }]{\text{transfer function}} \wh R(\la) \xleftrightarrow{\text{Lemma \ref{Lemma:from R to Rhat}}} R(\la) \xLeftarrow[\text{matrix (Lemma \ref{Lemma: from R to P})}]{\text{transfer function}} P(\la)
\end{equation*}

\subsection{One-sided factorizations}

Theorem \ref{thm:factorizations} is the only result in this subsection.
 It establishes a relation, in terms of one-side factorizations, between a rational matrix $R(\la)$ and the transfer function matrix $\wh R(\la)$ of a linearization for $R(\la)$ as in Theorem \ref{thm:linearization}.

\begin{theo}
\label{thm:factorizations} 
Let $R(\la)\in\F(\la)^{p\times m}$ be a rational matrix expressed in the form $R(\la)=D(\la)+C(\la)A(\la)^{-1}B(\la)$, for some nonsingular polynomial matrix  $A(\la)\in\FF[\la]^ {n\times n}$ and polynomial matrices
	$B(\la)\in\FF[\la]^{n\times m}$, $C(\la)\in\FF[\la]^{p\times n}$ and $D(\la)\in\FF[\la]^{p\times m}.$ Consider the linear polynomial matrix
	
\begin{equation*} 
	\mathcal{L}(\la) = \left[
	\begin{array}{c|c}
	M_A(\la) & \phantom{a} M_B(\la) \phantom{a} \\
	K_A(\la) & 0 \\ \hline \phantom{\Big|}
	-M_C(\la)  \phantom{\Big|}& M_D(\la) \\
	0 & K_D (\la)
	\end{array}
	\right],
\end{equation*} 
as in Theorem \ref{thm:linearization}, and let $\wh R(\la)$ be the transfer function matrix of $\mathcal{L}(\lambda)$.
Then, we have the following one-sided factorizations
	\begin{equation}\label{eq:right fact}
	\wh R(\la)N_{D}(\la)^{T}=
	\left[\begin{array}{c}
	R(\la) \\
	0 \end{array}\right],
	\end{equation}
and
\begin{equation}\label{eq:left fact}
\begin{bmatrix}
	I_p & -M_{R}(\la)\widehat{N}_D(\lambda)^T\end{bmatrix} \widehat{R}(\lambda) 
= R(\lambda) \widehat{K}_D(\lambda),
\end{equation}
where $M_{R}(\la):=M_D(\lambda)+C(\lambda)A(\lambda)^{-1}M_B(\lambda) $, and $\widehat{K}_D(\lambda)$ and $\widehat{N}_D(\lambda)$ are defined in \eqref{eq:embedding}.
\end{theo}
\begin{proof}
Notice that the matrix pencil $L_{A}(\la)=\left[\begin{smallmatrix}
	M_{A}(\la)  \\
	K_{A} (\lambda) \end{smallmatrix}\right]$ is regular since $L_{A}(\la)$ is a linearization of $A(\la)$ and the polynomial matrix $A(\la)$ is regular. 
Then, from 
	$\mathcal{L}_A(\lambda) N_{A}(\la)^{T} = \left[	\begin{smallmatrix}
	A(\la) \\
	0 \end{smallmatrix}\right]$,
we obtain
\begin{equation}\label{eq:inverse}
\left[	\begin{matrix}
M_{A}(\la)  \\
K_{A} (\lambda) \end{matrix}\right]^{-1} \left[	\begin{matrix}
I_n \\
0 \end{matrix}\right]=N_A(\la)^{T}A(\la)^{-1}.
\end{equation}
Hence, the transfer function matrix of $\mathcal{L}(\la)$ is given by
	\begin{align}
	\nonumber
	\wh R (\la)= & \left[	\begin{array}{c}
	M_{D}(\la)  \\
	K_{D} (\lambda) \end{array}\right] + 
	\left[	\begin{array}{c}
	M_{C}(\la)  \\
	0\end{array}\right] \left[	\begin{array}{c}
	M_{A}(\la)  \\
	K_{A} (\lambda) \end{array}\right]^{-1} 
	\left[	\begin{array}{c}
	M_{B}(\la)  \\
	0 \end{array}\right] = \\ \label{eq:R hat}
	&\begin{bmatrix}
	M_D(\lambda) + C(\lambda) A(\lambda)^{-1}M_B(\lambda) \\
	K_D(\lambda)
	\end{bmatrix},
\end{align}  
where we have used $M_{C}(\la)N_{A}(\la)^{T}=C(\la)$.

Multiplying \eqref{eq:R hat} on the right by $N_{D}(\la)^{T}$ yields
\begin{equation*}
	\wh R(\la)N_{D}(\la)^{T}  = \left[\begin{array}{c}
	D(\la)+ C(\la)A(\la)^{-1}B(\la)  \\
	0 \end{array}\right] = \left[	\begin{array}{c}
	R(\la) \\
	0 \end{array}\right],
	\end{equation*}
where we have used $M_{D}(\la)N_{D}(\la)^{T}=D(\la)$ and $M_{B}(\la)N_{D}(\la)^{T}=B(\la)$.
This establishes the right-sided factorization \eqref{eq:right fact}.

Multiplying \eqref{eq:R hat} on the left by $\left[\begin{smallmatrix}
I_p & -\left(M_D(\lambda)+C(\lambda)A(\lambda)^{-1}M_B(\lambda) \right)\widehat{N}_D(\lambda)^T\end{smallmatrix}\right]$  gives
\begin{align*}
&\begin{bmatrix}
	I_p & -\left(M_D(\lambda)+C(\lambda)A(\lambda)^{-1}M_B(\lambda) \right)\widehat{N}_D(\lambda)^T\end{bmatrix} \widehat{R}(\lambda) 
= \\
&\hspace{3cm} (M_D(\lambda)+C(\lambda)A(\lambda)^{-1}M_B(\lambda) )(I_{(\rho_D+1)m}-\widehat{N}_D(\lambda)^TK_D(\lambda)) = \\
&\hspace{3cm} (M_D(\lambda)+C(\lambda)A(\lambda)^{-1}M_B(\lambda) )N_D(\lambda)^T \widehat{K}_D(\lambda) = R(\lambda) \widehat{K}_D(\lambda),
\end{align*}
where $\rho_D = \deg N_D(\lambda)$. 
This establishes the left-sided factorization \eqref{eq:left fact}.
\end{proof}

\subsection{Recovery of minimal bases and minimal indices}

 In this section, we assume that the rational matrix $R(\lambda)$ is singular and show how to recover the right and left minimal indices and minimal bases of $R(\lambda)$ from those of a linearization $\mathcal{L}(\lambda)$ of $R(\lambda)$ as in Theorem \ref{thm:linearization}.

We begin with Lemma \ref{Lemma:from R to Rhat}, which establishes a bijection between the nullspaces of $R(\lambda)$ and the transfer function matrix of $\mathcal{L}(\lambda)$.


\begin{lem}\label{Lemma:from R to Rhat}
Let $R(\la)\in\F(\la)^{p\times m}$ be a rational matrix expressed in the form $R(\la)=D(\la)+C(\la)A(\la)^{-1}B(\la)$, for some nonsingular polynomial matrix  $A(\la)\in\FF[\la]^ {n\times n}$ and polynomial matrices $B(\la)\in\FF[\la]^{n\times m}$, $C(\la)\in\FF[\la]^{p\times n}$ and $D(\la)\in\FF[\la]^{p\times m}$.
Consider the linear polynomial matrix
\[
	\mathcal{L}(\la) = \left[
	\begin{array}{c|c}
	M_A(\la) & \phantom{a} M_B(\la) \phantom{a} \\
	K_A(\la) & 0 \\ \hline \phantom{\Big|}
	-M_C(\la)  \phantom{\Big|}& M_D(\la) \\
	0 & K_D (\la)
	\end{array}
	\right],
\]
as in Theorem \ref{thm:linearization}, and let $\wh R(\la)$ be the transfer function matrix of $\mathcal{L}(\lambda)$. Then the following statements hold:
\begin{enumerate}
\item[\rm (a)] The linear map
\begin{align*}
M_r \, : \, \mathcal{N}_r(R) & \longrightarrow \mathcal{N}_r(\widehat{R}) \\
 x(\lambda) & \longmapsto \widehat{x}(\lambda) := N_D(\lambda)^Tx(\lambda)
\end{align*}
is a bijection between the right nullspaces of $R(\lambda)$ and $\widehat{R}(\lambda)$. 
\item[\rm(b)] The linear map
\begin{align*}
M_\ell \,:\, \mathcal{N}_\ell(R) & \longrightarrow \mathcal{N}_\ell(\widehat{R}) \\
y(\lambda)^T & \longmapsto \widehat{y}(\lambda)^T := y(\lambda)^T \begin{bmatrix}
	I_p & -M_{R}(\la)\widehat{N}_D(\lambda)^T\end{bmatrix}
\end{align*}
is a bijection between the left nullspaces of $R(\lambda)$ and $\widehat{R}(\lambda)$, where $M_{R}(\la):=M_D(\lambda)+C(\lambda)A(\lambda)^{-1}M_B(\lambda) $ and $\widehat{N}_D(\lambda)$ is defined in \eqref{eq:embedding}. 
\end{enumerate}
\end{lem}
\begin{proof}
We will prove part (a).
Part (b) can be proved analogously.

That the map $M_r$ is a linear map from the right nullspace of $R(\lambda)$ to the right nullspace $\widehat{R}(\lambda)$ follows immediately from the right-sided factorization \eqref{eq:right fact}.
Moreover, by Theorem \ref{thm:linearization}, $\mathcal{L}(\lambda)$ is a linearization of $R(\lambda)$, at least, at some point $\la_0\in\F$. 
Indeed, since  $A(\la)$ is regular, there exists $\la_0\in\F$ such that $A(\la_0)$ is invertible. 
This implies that the realization $D(\la)+C(\la)A(\la)^{-1}B(\la)$ is minimal at $\la_0$.
Hence, we have $\dim \mathcal{N}_r(R) = \dim \mathcal{N}_r(\widehat{R})$.
Thus, to show that $M_r$ is a bijection, if suffices to show that it is injective.
So, suppose $\widehat{y}(\lambda) = N_D(\lambda)^Ty(\lambda) = 0$.
Since $N_D(\lambda)$ is a minimal basis, $N_D(\lambda)^T$ has full column rank.
Hence, $N_D(\lambda)^Ty(\lambda) = 0$ implies $y(\lambda)=0$.
This establishes the injectivity of the linear map $M_r$.
\end{proof}

\begin{rem}\rm
 Since the maps in Lemma \ref{Lemma:from R to Rhat} are bijections, they preserve linear independence and allow us to recover bases of the right (resp. left) nullspace of $R(\la)$ from bases of the right (resp. left) nullspace of $\wh R(\la)$, and conversely. 
\end{rem}

	
 By combining Lemmas \ref{Lemma: from R to P} and \ref{Lemma:from R to Rhat}, we obtain Theorem \ref{thm:from R to L}, which establishes a bijection between the nullspaces of $R(\lambda)$ and $\mathcal{L}(\lambda)$.
	\begin{theo}\label{thm:from R to L}
		Let $R(\la)\in\F(\la)^{p\times m}$ be a rational matrix expressed in the form $R(\la)=D(\la)+C(\la)A(\la)^{-1}B(\la)$, for some nonsingular polynomial matrix  $A(\la)\in\FF[\la]^ {n\times n}$ and polynomial matrices $B(\la)\in\FF[\la]^{n\times m}$, $C(\la)\in\FF[\la]^{p\times n}$ and $D(\la)\in\FF[\la]^{p\times m}$.
		Consider the linear polynomial matrix
		\[
		\mathcal{L}(\la) = \left[
		\begin{array}{c|c}
		M_A(\la) & \phantom{a} M_B(\la) \phantom{a} \\
		K_A(\la) & 0 \\ \hline \phantom{\Big|}
		-M_C(\la)  \phantom{\Big|}& M_D(\la) \\
		0 & K_D (\la)
		\end{array}
		\right],
		\]
		as in Theorem \ref{thm:linearization}, and let $\wh R(\la)$ be the transfer function matrix of $\mathcal{L}(\lambda)$. Then the following statements hold:
		\begin{enumerate}
			\item[\rm (a)] The linear map
			\begin{align*}
			F_r \, : \, \mathcal{N}_r(R) & \longrightarrow \mathcal{N}_r(\mathcal{L}) \\
			x(\lambda) & \longmapsto \widetilde{x}(\lambda) :=\begin{bmatrix}
			 -N_{A}(\la)^{T}A(\la)^{-1}B(\la)\\
             N_D(\lambda)^T
			\end{bmatrix} x(\la)
			\end{align*}
			is a bijection between the right nullspaces of $R(\lambda)$ and $\mathcal{L}(\lambda)$. 
			\item[\rm(b)] The linear map
			\begin{align*}
			F_\ell \,:\, \mathcal{N}_\ell(R) & \longrightarrow \mathcal{N}_\ell(\mathcal{L}) \\
			y(\lambda)^T & \longmapsto \widetilde{y}(\lambda)^T := y(\lambda)^T \begin{bmatrix}
			M_{C}(\la)L_{A}(\la)^{-1} & I_p & -M_{R}(\la)\widehat{N}_D(\lambda)^T\end{bmatrix}
			\end{align*}
			is a bijection between the left nullspaces of $R(\lambda)$ and $\mathcal{L}(\lambda)$, where $L_{A}(\la):=\begin{bmatrix}
			M_{A}(\la)\\
			K_{A}(\la)
			\end{bmatrix}$, $M_{R}(\la):=M_D(\lambda)+C(\lambda)A(\lambda)^{-1}M_B(\lambda) $ and $\widehat{N}_D(\lambda)$ is defined in \eqref{eq:embedding}. 
		\end{enumerate}
	\end{theo}
	\begin{proof}  We will prove part (a).
		Part (b) can be proved analogously.
		
		Consider the linear bijections $T_r$ and $M_r$ in Lemma \ref{Lemma: from R to P} and Lemma \ref{Lemma:from R to Rhat}, respectively. 
		Then, $F_r$ is the composition $F_r=T_r\circ M_r$. 
		Indeed, we have
		\begin{align*}
		F_r \, : \, \mathcal{N}_r(R) & \longrightarrow \mathcal{N}_r(\wh{R}) & \longrightarrow  &\,\, \mathcal{N}_r(\mathcal{L}) \\
		x(\lambda) & \longmapsto  N_D(\lambda)^Tx(\lambda) &  \longmapsto & \left[\begin{array}{c}
		-\left[\begin{array}{c}
		M_{A}(\la)  \\
		K_{A} (\lambda) 	\end{array} \right]^{-1}\left[\begin{array}{c}
		M_B(\la) \\
		0 	\end{array} \right]N_D(\lambda)^Tx(\lambda)   \\ 
		N_D(\lambda)^Tx(\lambda)
		\end{array} \right] = \\
		&&&  \begin{bmatrix}
		-N_A(\lambda)^TA(\lambda)^{-1}B(\lambda) x(\lambda) \\
		N_D(\lambda)^Tx(\lambda)
		\end{bmatrix},
		\end{align*}
		where we have used \eqref{eq:inverse} and $M_{B}(\la)N_D(\la)^T=B(\la)$. 
		
	\end{proof}


 Lemma \ref{Lemma:degrees} will allow us to prove that, under some minimality conditions, the linear polynomial system matrix $\mathcal{L}(\lambda)$ and its transfer function $\widehat{R}(\lambda)$ have the same right and left minimal indices.
	\begin{lem}\label{Lemma:degrees}
		Consider a linear polynomial system matrix 
	\begin{equation*}
		L(\la)=\left[\begin{array}{c|c}
		A_1 \la +A_0 &B_1 \la +B_0\\\hline \phantom{\Big|} -(C_1 \la +C_0)&D_1 \la +D_0
		\end{array}\right]\in\F[\la]^{(n+q)\times (n+r)},
		\end{equation*} with state matrix $A_1 \la +A_0\in\mathbb{F}[\lambda]^{n\times n}$ and transfer function matrix $T(\la).$ Then the following statements hold.
		\begin{itemize}
			\item[\rm(a)] If $u(\la)=\begin{bmatrix}y(\la)\\x(\la)\end{bmatrix}\in\mathcal{N}_{r}(L)$, then $x(\la)\in\mathcal{N}_{r}(T).$ 
			In addition, if $\rank\begin{bmatrix} A_1 \\ C_1 \end{bmatrix}=n$ and $u(\la)$ is a polynomial vector, then $\deg u(\la)=\deg x(\la).$
			\item[\rm(b)] If $v(\la)^T=\begin{bmatrix}\tilde y(\la)^T & \tilde x(\la)^T\end{bmatrix}\in\mathcal{N}_{\ell}(L)$, then $\tilde x(\la)^T\in\mathcal{N}_{\ell}(T).$ 
			In addition, if $\rank\begin{bmatrix} A_1 & B_1 \end{bmatrix}=n$ and $v(\la)$ is a polynomial vector, then $\deg v(\la)=\deg \tilde x(\la).$
			
		\end{itemize}
	\end{lem}

\begin{rem}\rm  We emphasize that
	\begin{equation}\label{min_infinity}
\rank\begin{bmatrix} A_1 \\ C_1 \end{bmatrix}=\rank\begin{bmatrix} A_1 & B_1 \end{bmatrix}=n
	\end{equation}
	is the condition for a linear polynomial system matrix to be minimal at infinity in the sense of \cite{local}, which is also a necessary condition for a linear polynomial system matrix to be a linearization at infinity \cite{local}. 
In \cite{linearsystems} there is a procedure to reduce any linear polynomial system matrix to one satisfying condition \eqref{min_infinity}.
\end{rem}
	
\begin{proof} {\rm (of Lemma \ref{Lemma:degrees})}
 We only prove part (a) since part (b) follows from a similar argument. From Lemma \ref{Lemma: from R to P}, we obtain that if $u(\lambda)\in\mathcal{N}_r(L)$, then the vector $u(\lambda)$ must be of the form 
\[
u(\lambda) =
\begin{bmatrix}
y(\lambda) \\ x(\lambda)
\end{bmatrix} =
\begin{bmatrix} -(A_1\lambda+A_0)^{-1}(B_1\lambda + B_0)x(\lambda)\\x(\lambda) \end{bmatrix},
\]
for some $x(\lambda)\in\mathcal{N}_r(T)$, as we wanted to show.

For proving that $\deg u(\lambda) = \deg x(\lambda)$, we will show that $\deg y(\lambda)\leq  \deg x(\lambda)$ by contradiction.
Let us assume that $\ell:=\deg y(\lambda) > \deg x(\lambda)$. 
Then, the vector $u(\lambda)$ must be of the form
\[
u(\lambda)= \begin{bmatrix} y_\ell \\ 0 \end{bmatrix} \lambda^\ell + \mbox{ lower order terms}, \quad \mbox{with }y_\ell\neq 0.
\]
Since $u(\la)\in\mathcal{N}_{r}(L),$ we have
\begin{align*}
&( A_1 \la + A_0)y(\la)+( B_1 \la + B_0)x(\la)=0, \\
&( C_1 \la + C_0)y(\la)-( D_1 \la + D_0)x(\la)=0.
\end{align*}
Considering the highest degree terms in the left hand side of the two equations above, we obtain
		\begin{equation*}
		\begin{bmatrix}A_1\\C_1\end{bmatrix}y_\ell=0.
		\end{equation*}
Since the matrix $\left[\begin{smallmatrix} A_1 \\ C_1 \end{smallmatrix}\right]$ has full column rank by assumption, we get $y_\ell=0$, which contradicts our original hypothesis.
\end{proof}

For completeness, in Lemma \ref{Lemma:ranks} we recall \cite[Lemma 2]{VVK79} and different versions of it that can be analogously proved.

	 \begin{lem}\label{Lemma:ranks} 
	 	Let $\begin{bmatrix}X_1&X_2\end{bmatrix}\begin{bmatrix}Y_1 \\Y_2\end{bmatrix}=0$. 
	 	\begin{itemize}
	 		\item[\rm(a)]Assume that $\begin{bmatrix}Y_1 \\Y_2\end{bmatrix}$ has full column rank.
	 		\begin{itemize}
	 			\item[\rm(a1)] If $X_1$ has full column rank, then $Y_2$ also has full column rank.
	 			\item[\rm(a2)] If $X_2$ has full column rank, then $Y_1$ also has full column rank.
	 	\end{itemize}	
	 		\item[\rm(b)]Assume that $\begin{bmatrix}X_1&X_2\end{bmatrix}$ has full row rank.
	 			\begin{itemize}
	 			\item[\rm(b1)] If $Y_1$ has full row rank, then $X_2$ also has full row rank.
	 			\item[\rm(b2)] If $Y_2$ has full row rank, then $X_1$ also has full row rank.
	 		\end{itemize}	
	 	\end{itemize}
	 \end{lem} 
 
We are finally ready to state and prove the main results of this section, Theorems \ref{thm:right_minimalindices} and \ref{thm:left_minimalindices}.
These theorems show how to recover right and left minimal bases and minimal indices of rational matrices from those of their linearizations in Theorem \ref{thm:linearization}. 
\begin{theo}[Right minimal bases and minimal indices]\label{thm:right_minimalindices}
Let $R(\la)\in\F(\la)^{p\times m}$ be a rational matrix expressed in the form $R(\la)=D(\la)+C(\la)A(\la)^{-1}B(\la)$, for some nonsingular polynomial matrix  $A(\la)\in\FF[\la]^ {n\times n}$ and polynomial matrices
		$B(\la)\in\FF[\la]^{n\times m}$, $C(\la)\in\FF[\la]^{p\times n}$ and $D(\la)\in\FF[\la]^{p\times m}$.
Consider the linear polynomial matrix 
\begin{equation*} 
		\mathcal{L}(\la) = \left[
		\begin{array}{c|c}
		M_A(\la) & \phantom{a} M_B(\la) \phantom{a} \\
		K_A(\la) & 0 \\ \hline \phantom{\Big|}
		-M_C(\la)  \phantom{\Big|}& M_D(\la) \\
		0 & K_D (\la)
		\end{array}
		\right]\in\mathbb{F}[\lambda]^{(n+p+\rho_An+\rho_Dm)\times(n+m+\rho_An+\rho_Dm)},
\end{equation*}
as in Theorem \ref{thm:linearization}, where $	L_{A}(\la)=
		\left[
		\begin{smallmatrix}
		M_{A}(\la)  \\
		K_{A} (\lambda) 
		\end{smallmatrix}
		\right] $ and $L_{D}(\la) =
		\left[
		\begin{smallmatrix}
		M_{D}(\la)  \\
		K_{D} (\lambda) 
		\end{smallmatrix}
		\right]
		$ are strong block minimal basis pencils associated with $A(\la)$ and $D(\la),$ respectively, and $\rho_i:=\deg N_{i}(\la)$, for $i\in\{A,D\}$.
Let $\wh R(\la)$ be the transfer function matrix of $\mathcal{L}(\la)$.
If
\begin{align}
\label{eq:right min ind 1}
&\rank \begin{bmatrix} A(\la_0) \\ C(\la_0) \end{bmatrix}=n \quad \mbox{for all} \quad \la_0\in\F, \quad \mbox{and} \\
\label{eq:right min ind 2}
&\rank\begin{bmatrix}\rev_{\rho_A +1}A(0)\\ \rev_{\rho_A +1}C(0) \end{bmatrix}=n,
\end{align}
then the following statements hold:
		\begin{itemize}
\item[\rm (a)] If $\left\{\left[\begin{smallmatrix}y_{i}(\la)\\x_{i}(\la)\end{smallmatrix}\right]\right\}_{i=1}^{s}$ is a right minimal basis of $\mathcal{L}(\lambda)$, where $x_i(\lambda)\in\mathbb{F}[\lambda]^{(\rho_D+1)m}$, then $\{x_i(\la)\}_{i=1}^{s}$ is a right minimal basis of  $\widehat{R}(\lambda)$, and there exists a right minimal basis $\{u_{i}(\la)\}_{i=1}^{s}$ of $R(\lambda)$ such that $x_{i}(\la)=N_D(\la)^Tu_i(\la)$, for $i=1,\hdots,s$.
			
\item[\rm (b)] If $\epsilon_1\leq \cdots \leq \epsilon_s$ are the right minimal indices of $\mathcal{L}(\la)$, then $\epsilon_1\leq \cdots \leq \epsilon_s$ are the right minimal indices of $\wh R(\la)$, and $\epsilon_1-\rho_D\leq \cdots \leq \epsilon_s-\rho_D$ are the right minimal indices of $R(\la)$.
		\end{itemize}
\end{theo}
\begin{proof} See \ref{app:right}.
\end{proof}

\begin{rem}\rm
We recall that the polynomial matrix $\widehat{K}_D(\lambda)$ in \eqref{eq:embedding} is the left polynomial matrix inverse of $N_D(\lambda)^T$.
Hence, from the right minimal basis $\{x_i(\lambda)\}_{i=1}^s$ of $\widehat{R}(\lambda)$ in part (a) of Theorem \ref{thm:right_minimalindices}, we can recover a right minimal basis of the rational matrix $R(\lambda)$ as $\{u_i(\lambda)\}_{i=1}^s=\{\widehat{K}_D(\lambda)x_i(\lambda)\}_{i=1}^s$.
\end{rem}

\begin{theo}[Left minimal bases and minimal indices]\label{thm:left_minimalindices}
Let $R(\la)\in\F(\la)^{p\times m}$ be a rational matrix expressed in the form $R(\la)=D(\la)+C(\la)A(\la)^{-1}B(\la)$, for some nonsingular polynomial matrix  $A(\la)\in\FF[\la]^ {n\times n}$ and polynomial matrices
		$B(\la)\in\FF[\la]^{n\times m}$, $C(\la)\in\FF[\la]^{p\times n}$ and $D(\la)\in\FF[\la]^{p\times m}$.
Consider the linear polynomial matrix 
\begin{equation*} 
		\mathcal{L}(\la) = \left[
		\begin{array}{c|c}
		M_A(\la) & \phantom{a} M_B(\la) \phantom{a} \\
		K_A(\la) & 0 \\ \hline \phantom{\Big|}
		-M_C(\la)  \phantom{\Big|}& M_D(\la) \\
		0 & K_D (\la)
		\end{array}
		\right]\in\mathbb{F}[\lambda]^{(n+p+\rho_An+\rho_Dm)\times(n+m+\rho_An+\rho_Dm)},
\end{equation*}
as in Theorem \ref{thm:linearization}, where $	L_{A}(\la)=
		\left[
		\begin{smallmatrix}
		M_{A}(\la)  \\
		K_{A} (\lambda) 
		\end{smallmatrix}
		\right] $ and $L_{D}(\la) =
		\left[
		\begin{smallmatrix}
		M_{D}(\la)  \\
		K_{D} (\lambda) 
		\end{smallmatrix}
		\right]
		$ are strong block minimal basis pencils associated with $A(\la)$ and $D(\la),$ respectively, and $\rho_i:=\deg N_{i}(\la)$, for $i\in\{A,D\}$.
Let $\wh R(\la)$ be the transfer function matrix of $\mathcal{L}(\la)$.
If
\begin{align}\label{eq:left min ind 1}
&\rank \begin{bmatrix} A(\la_0) & B(\la_0) \end{bmatrix}=n \quad \mbox{for all} \quad \la_0\in\F, \quad \mbox{and}  \\
\label{eq:left min ind 2}
&\rank\begin{bmatrix}\rev_{\rho_A +1}A(0) & \rev_{\rho_D +1}B(0) \end{bmatrix}=n,
\end{align}
then the following statements hold:
\begin{itemize}
\item[\rm(a)] If $\{z_i(\lambda)^T\}_{i=1}^t$ is a left minimal basis of $\mathcal{L}(\lambda)$, then $z_i(\lambda)^T = \begin{bmatrix} y_i(\lambda)^T & x_i(\lambda)^T \end{bmatrix}$, for $i=1,\hdots,s$, for some left minimal basis $\{x_i(\lambda)^T\}_{i=1}^s$ of $\widehat{R}(\lambda)$, and $x_i(\lambda)^T = \begin{bmatrix} u_i(\lambda)^T & w_i(\lambda)^T \end{bmatrix}$, for $i=1,\hdots,s$, for some left minimal basis $\{u_i(\lambda)^T\}_{i=1}^s$ of $R(\lambda)$.
\item[\rm(b)] If $\eta_1\leq \cdots \leq \eta_t$ are the left minimal indices of $\mathcal{L}(\lambda)$, then  $\eta_1\leq \cdots \leq \eta_t$ are the left minimal indices of $\widehat{R}(\lambda)$ and $R(\lambda)$.
\end{itemize}
\end{theo}

\begin{proof}
See \ref{app:left}.
\end{proof}

\subsection{Recovery of eigenvectors}

 In this section, we assume that the rational matrix $R(\lambda)$ is regular and  show how to recover right and left eigenvectors of $R(\lambda)$ from those of a linearization $\mathcal{L}(\lambda)$ of $R(\lambda)$ as in Theorem \ref{thm:linearization}.

For a finite eigenvalue $\lambda_{0}\in\F$ of a rational matrix $R(\la)$, we denote by $\mathcal{N}_r (R(\la_{0}))$ and  $\mathcal{N}_{\ell} (R(\la_{0}))$ the right and left nullspaces over $\F$ of  the constant matrix $R(\la_{0})$, respectively. 
More precisely, we have

\[
\begin{array}{l}
\mathcal{N}_r (R(\la_0))=\{x\in\FF^{m\times 1}: R(\la_0)x=0\}, \text{ and}\\
\mathcal{N}_\ell (R(\la_0))=\{y^T\in\FF^{1\times p}: y^T R(\la_0)=0\}.

\end{array}
\]

 In Proposition \ref{prop:eigenvectors} we state, without proof, analogous results to those of Theorem \ref{thm:from R to L} but for the right and left nullspaces of $R(\la)$ evaluated at a particular value $\la_0$.

\begin{prop}\label{prop:eigenvectors}  
	Let $R(\la)\in\F(\la)^{p\times m}$ be a rational matrix expressed in the form $R(\la)=D(\la)+C(\la)A(\la)^{-1}B(\la)$, for some nonsingular polynomial matrix  $A(\la)\in\FF[\la]^ {n\times n}$ and polynomial matrices
	$B(\la)\in\FF[\la]^{n\times m}$, $C(\la)\in\FF[\la]^{p\times n}$ and $D(\la)\in\FF[\la]^{p\times m}.$ Consider the linear polynomial matrix \begin{equation*} 
	\mathcal{L}(\la) = \left[
	\begin{array}{c|c}
	M_A(\la) & \phantom{a} M_B(\la) \phantom{a} \\
	K_A(\la) & 0 \\ \hline \phantom{\Big|}
	-M_C(\la)  \phantom{\Big|}& M_D(\la) \\
	0 & K_D (\la)
	\end{array}
	\right],
	\end{equation*} in Theorem \ref{thm:linearization}. Let $\la_0\in\F$ such that $\det A(\la_0)\neq 0$, then the following statements hold:
	
	\begin{enumerate}
		\item[\rm (a)] The linear map
		\begin{align*}
		F_r \, : \, \mathcal{N}_r(R(\la_0)) & \longrightarrow \mathcal{N}_r(\mathcal{L}(\la_0)) \\
		x & \longmapsto \widetilde{x}:=\begin{bmatrix}
		-N_{A}(\la_0)^{T}A(\la_0)^{-1}B(\la_0)\\
		N_D(\la_0)^T
		\end{bmatrix} x
		\end{align*}
		is a bijection between the right nullspaces over $\F$ of $R(\la_0)$ and $\mathcal{L}(\la_0)$. 
		\item[\rm(b)] The linear map
		\begin{align*}
		F_\ell \,:\, \mathcal{N}_\ell(R(\la_0)) & \longrightarrow \mathcal{N}_\ell(\mathcal{L}(\la_0)) \\
		y^T & \longmapsto \widetilde{y}^T := y^T \begin{bmatrix}
		M_{C}(\la_0)L_{A}(\la_0)^{-1} & I_p & -M_{R}(\la_0)\widehat{N}_D(\la_0)^T\end{bmatrix}
		\end{align*}
		is a bijection between the left nullspaces over $\F$ of $R(\la_0)$ and $\mathcal{L}(\la_0)$, where $L_{A}(\la):=\begin{bmatrix}
		M_{A}(\la)\\
		K_{A}(\la)
		\end{bmatrix}$, $M_{R}(\la):=M_D(\lambda)+C(\lambda)A(\lambda)^{-1}M_B(\lambda) $ and $\widehat{N}_D(\lambda)$ is defined in \eqref{eq:embedding}. 
	\end{enumerate}
\end{prop}

\begin{rem}\rm
Let $\mathcal{L}(\lambda)$ be as in Proposition \ref{prop:eigenvectors}, let $\lambda_0\in\mathbb{F}$ be an eigenvalue of $\mathcal{L}(\lambda)$ such that $\det A(\lambda_0)\neq 0$, and let $\widetilde{x}$ and $\widetilde{y}^T$ be, respectively, right and left eigenvectors of $\mathcal{L}(\lambda)$ with eigenvalue $\lambda_0$.

By Proposition \ref{prop:eigenvectors}, the vector $\widetilde{y}^T$ must be of the form
\[
\widetilde{y}^T = y^T \begin{bmatrix}
		M_{C}(\la_0)L_{A}(\la_0)^{-1} & I_p & -M_{R}(\la_0)\widehat{N}_D(\la_0)^T\end{bmatrix},
\]
for some left eigenvector of $R(\lambda)$ with eigenvalue $\lambda_0$.
Hence, one can readily recover a left eigenvector $y^T$ of $R(\lambda)$ from the middle block of $\widetilde{y}^T$.
Furthermore, from Proposition \ref{prop:eigenvectors}, we get that $\widetilde{x}$ must be of the form
\[
\widetilde{x}=\begin{bmatrix}
		-N_{A}(\la_0)^{T}A(\la_0)^{-1}B(\la_0)\\
		N_D(\la_0)^T
		\end{bmatrix}x,
\]
for some right eigenvector of $R(\lambda)$ with eigenvalue $\lambda_0$.
Since the polynomial matrix $\widehat{K}_D(\lambda)$ in \eqref{eq:embedding} satisfies $\widehat{K}_D(\lambda)N_D(\lambda)^T=I_{m}$ for all $\lambda\in\mathbb{F}$, we have $\widehat{K}_D(\lambda_0)N_D(\lambda_0)^Tx = x$.
Thus, one can also recover a right eigenvector of $R(\lambda)$ from the right eigenvector $\widetilde{x}$ of $\mathcal{L}(\lambda)$.
\end{rem}

 \section{Application to scalar rational equations}\label{sec:scalar}
%
 
In this section, we show by example how the theory developed in this paper can be used for solving (scalar) rational equations of the form
 	\begin{equation}\label{rational_scalar}
 	\dfrac{c(\la)}{a(\la)}=\dfrac{d(\la)}{b(\la)},
 	\end{equation}
 	where $a(\la)$, $b(\la)$, $c(\la)$ and $d(\la)$ are nonzero scalar polynomials, and where the numerators and the denominators of each rational function can be expressed in terms of different polynomial bases. 
For instance, let us assume that the polynomials $a(\la)$ and $c(\la)$ are written in terms of the monomial basis, that is, 
\[
a(\la)=\displaystyle\sum_{i=0}^{n}a_i\la^{i}
\qquad \mbox{and} \qquad 
c(\la)=\displaystyle\sum_{i=0}^{n}c_i\la^{i},
\] with $n=\max\{ \mathrm{deg}\,a(\la), \mathrm{deg}\,c(\la)\}$, and that the polynomials $b(\la)$ and $d(\la)$ are written in terms of Chebyshev polynomials of the first kind $\{\phi_{j}(\lambda)\}_{j=0}^{\infty}$, that is,  
\[
b(\la)=\displaystyle\sum_{i=0}^{m}b_i\phi_{i}(\lambda)
\qquad \mbox{and} \qquad 
d(\la)=\displaystyle\sum_{i=0}^{m}d_i\phi_{i}(\lambda),
\]
with $m=\max\{ \mathrm{deg}\,b(\la), \mathrm{deg}\,d(\la)\}$. 
We recall that the Chebyshev basis $\{\phi_{j}(\lambda)\}_{j=0}^{\infty}$  satisfies the three-term recurrence relation:
 	\begin{equation}\label{eq:three-term}
 	\frac{1}{2}\phi_{j+1}(\lambda)=\lambda\phi_{j}(\lambda)-\frac{1}{2}\phi_{j-1}(\lambda) \quad j\geq 1
 	\end{equation}
 	where $\phi_{-1}(\lambda)=0,$ $\phi_{0}(\lambda)=1$ and $\phi_{1}(\lambda)=\lambda.$

Notice that, outside the set of the roots of $b(\la)$, that is, in $\Omega:=\C\setminus \{\la_0\in\C : b(\la_0)=0\}$, equation \eqref{rational_scalar} is equivalent to the equation 
	\begin{equation}\label{rational_scalar2}
r(\la):=d(\la)-c(\la)a(\la)^{-1}b(\la)=0.
\end{equation}
For computing the roots of \eqref{rational_scalar2}, that is, the zeros that are not poles, we consider a linear polynomial matrix system of the form
\begin{equation*} 
\mathcal{L}(\la) = \left[
\begin{array}{c|c}
M_a(\la) & \phantom{a} M_b(\la) \phantom{a} \\
K_a(\la) & 0 \\ \hline \phantom{\Big|}
-M_c(\la)  \phantom{\Big|}& M_d(\la) \\
0 & K_d (\la)
\end{array}
\right],
\end{equation*}
where
\begin{align*}
&M_{a}(\lambda):=\left[  a_{n}\lambda+a_{n-1}\quad a_{n-2}\quad a_{n-3}\quad \cdots\quad a_{1}\quad a_{0}\right],\\
&M_{c}(\lambda):=\left[  c_{n}\lambda+c_{n-1}\quad c_{n-2}\quad c_{n-3}\quad \cdots\quad c_{1}\quad c_{0}\right],
\end{align*}
and 
	\begin{equation*}\label{K_a}
	K_a(\lambda):=
	\left[ {\begin{array}{ccccc}
		-1 & \lambda & 0 & \cdots & 0 \\
		0 & -1 & \lambda  & \ddots & \vdots \\
		\vdots & \ddots &\ddots &\ddots & 0  \\
		0 & \cdots & 0 &-1& \lambda 
		\end{array} } \right] \qquad \mbox{and} \qquad
	N_a(\lambda)^T = 
	\begin{bmatrix}
		\lambda^{n-1}\\ \vdots \\ \lambda \\ 1
\end{bmatrix}	 
	\end{equation*}
	is a pair of dual minimal bases, and
\begin{align*}
&M_{b}(\lambda):=\left[ 2 b_{m} \lambda +b_{m-1}\quad b_{m-2}-b_{m}\quad b_{m-3}\quad \cdots\quad b_{1}\quad b_{0}\right],\\
&M_{d}(\lambda):=\left[ 2 d_{m} \lambda +d_{m-1}\quad d_{m-2}-d_{m}\quad d_{m-3}\quad \cdots\quad d_{1}\quad d_{0}\right],
\end{align*}
and, by \eqref{eq:three-term},
			\begin{equation*}\label{may}
			K_{d}(\lambda)=
			\left[ {\begin{array}{cccccc}
				-\frac{1}{2} & \lambda  & -\frac{1}{2} &  0 & \cdots & 0 \\
				0 & -\frac{1}{2} & \lambda  & -\frac{1}{2} & \ddots & \vdots \\
				\vdots & \ddots &\ddots &\ddots & \ddots & 0 \\
				\vdots & & \ddots &-\frac{1}{2} & \lambda  & -\frac{1}{2}  \\
				0 & \cdots & \cdots & 0 &-1& \lambda
				\end{array} } \right]  \qquad \mbox{and} \qquad
			N_d(\lambda)^T = 
			\begin{bmatrix}
				\phi_{m-1}(\lambda) \\ \phi_{m-2}(\lambda) \\ \vdots \\\phi_1(\lambda) \\ \phi_0(\lambda)
			\end{bmatrix}
			\end{equation*}
is another pair of dual minimal bases.
Observe that $a(\lambda)=M_a(\lambda)N_a(\lambda)^T$, $c(\lambda)=M_c(\lambda)N_a(\lambda)^T$, $b(\lambda)=M_b(\lambda)N_d(\lambda)^T$ and $d(\lambda)=M_d(\lambda)N_d(\lambda)^T$.

It is immediate that the  matrices
\[
\begin{bmatrix}
	a(\lambda_0)\\c(\lambda_0)
\end{bmatrix} \qquad \mbox{and} \qquad
\begin{bmatrix}
	a(\lambda_0) & b(\lambda_0)
\end{bmatrix}
\]
have full rank (equal to 1) at every $\la_0$ that is not a root of $a(\la)$ and $c(\la)$ simultaneously.
Hence, if $\frac{c(\la)}{a(\la)}$ is irreducible, i.e., $a(\la)$ and $c(\la)$ do not have roots in common, then, by Theorem \ref{thm:linearization}, $\mathcal{L}(\la)$ is a linearization of $r(\la)$ in $\Omega$. 
Therefore, the zeros of $\mathcal{L}(\la)$ in $\Omega$ are the zeros of $r(\la)$ in $\Omega$.

The idea of transforming the rational problem \eqref{rational_scalar} into an eigenvalue problem is not new \cite{Leo1}.
An algorithm based on the Ehrlich-Aberth iteration that uses this approach can be found in \cite{Leo2}.
\section{Conclusions and future work}\label{conclusion}

Associated with a rational matrix $R(\lambda)$ expressed in the general form $R(\lambda)=D(\lambda)+C(\lambda)A(\lambda)^{-1}B(\lambda)$, we have constructed a family of linear polynomial matrices that, under some minimality conditions, are local linearizations for $R(\lambda)$ in the sense defined in \cite{local}.
Unlike other lineariations for rational matrices recently introduced \cite{AlBe16, strong,dopmarquin2019, SuBai}, ours do not require neither to decompose $R(\lambda)$ into its polynomial and strictly proper rational parts nor to express the strictly proper rational part in state-space form. 
Moreover, we have showed how to recover the eigenvectors of $R(\lambda)$, when $R(\lambda)$ is regular, and the minimal bases and minimal indices of $R(\lambda)$, when $R(\lambda)$ is singular, from those of any of the new linearizations.

We are currently applying the theory developed in this paper to three different problems: (i) we are studying the stability and accuracy of linearization-based algorithms for solving (scalar) rational equations as in \eqref{rational_scalar}; (ii) we are developing novel techniques to linearize product and composition polynomial matrices $A(\lambda)B(\lambda)$ and  $A(B(\lambda))$, where $A(\lambda)$ and $B(\lambda)$ are polynomial matrices, without ever taking the product or the composition; and (iii) we are building trimmed linearizations for polynomial matrices that automatically deflate eigenvalues at zero or/and at infinity without any computational cost.

\appendix

\section{Proof of Theorem \ref{thm:right_minimalindices}}\label{app:right}

\begin{proof}
	Let us consider a right minimal basis $\{z_i(\lambda)\}_{i=1}^s$ of $\mathcal{L}(\lambda)$ with right minimal indices $\epsilon_i=\deg z_i(\lambda)$, for $i=1,\hdots,s$.
	By Lemma \ref{Lemma: from R to P}, we have that the polynomial vectors $z_i(\lambda)$ must be of the form
	\[
	z_i(\lambda) = 
	\begin{bmatrix}
	y_i(\lambda)\\
	x_i(\lambda) 
	\end{bmatrix} \quad (i=1,\hdots,s),
	\]
	for some basis $\{x_i(\lambda)\}_{i=1}^s$ of $\mathcal{N}_r(\widehat{R})$.
	We notice that the vectors $x_i(\lambda)$ must be polynomial vectors, otherwise the vectors $z_i(\lambda)$ would not be polynomial vectors.
	
	We will prove that the polynomial basis $\{x_i(\lambda)\}_{i=1}^s$ is minimal by using Theorem \ref{Thm:charac min bases}.
	For this purpose, let us define the polynomial matrices
	\[
	B(\lambda) :=
	\begin{bmatrix}
	z_1(\lambda) & \cdots & z_s(\lambda)
	\end{bmatrix} \quad \mbox{and} \quad 
	\widehat{B}(\lambda) := 
	\begin{bmatrix}
	x_1(\lambda) & \cdots & x_s(\lambda)
	\end{bmatrix}.
	\]
	
	First, let us show that $\widehat{B}(\lambda)$ has full column rank for every $\lambda_0\in\mathbb{F}$.
	Consider the unimodular (and, so, invertible at every $\lambda_0\in\mathbb{F}$) matrix $U_A(\lambda)^{-1} = \begin{bmatrix} \widehat{N}_A(\lambda)^T & N_A(\lambda)^T \end{bmatrix}$ defined in \eqref{eq:embedding}.
	We have
	\[
	\begin{bmatrix}
	M_A(\lambda_0) \\
	K_A(\lambda_0) \\
	-M_C(\lambda_0) \\ 
	0
	\end{bmatrix} U_A(\lambda_0)^{-1} =
	\begin{bmatrix}
	* & A(\lambda_0) \\
	I_{n\,\rho_A} & 0 \\
	* & C(\lambda_0) \\
	0 & 0 
	\end{bmatrix},
	\]
	for every $\lambda_0\in\mathbb{F}$.
	The equation above, together with \eqref{eq:right min ind 1}, implies that the (constant) matrix 
	\[
	\left[\begin{matrix} M_A(\lambda_0) \\ K_A(\lambda_0) \\ -M_C(\lambda_0) \\0 \end{matrix} \right]
	\]
	has full column rank for every $\lambda_0\in\mathbb{F}$.
	Then, from $\mathcal{L}(\lambda_0)B(\lambda_0)=0$, we obtain
	\[
	\left[\begin{array}{c|c}
	M_A(\lambda_0) & M_B(\lambda_0) \\
	K_A(\lambda_0) & 0 \\
	-M_C(\lambda_0) & M_D(\lambda_0)\\
	0 & K_D(\lambda_0) 
	\end{array}\right]
	\begin{bmatrix}
	* \\ 
	\widehat{B}(\lambda_0)
	\end{bmatrix} = 0,
	\]
	where $*$ indicates a constant matrix that is not important for the argument.
	By Lemma \ref{Lemma:ranks}, we conclude that $\widehat{B}(\lambda_0)$ has full column rank.
	
	Let us consider the highest column degree coefficient matrices of $B(\lambda)$ and $\widehat{B}(\lambda)$, which we denote by $B_{\rm hcd}$ and $\widehat{B}_{\rm hcd}$, respectively.
	Let us show, next, that the matrix $\widehat{B}_{\rm hcd}$ has full column rank.
	For this purpose, let us write
	\begin{equation}\label{eq:notation L}
	\mathcal{L}(\la) = \left[
	\begin{array}{c|c}
	M_A(\la) & \phantom{a} M_B(\la) \phantom{a} \\
	K_A(\la) & 0 \\ \hline \phantom{\Big|}
	-M_C(\la)  \phantom{\Big|}& M_D(\la) \\
	0 & K_D (\la)
	\end{array}
	\right]=: \left[
	\begin{array}{c|c}
	M_{1A}\la + M_{0A} & \phantom{a} M_{1B}\la + M_{0B}\phantom{a} \\
	K_{1A}\la + K_{0A}& 0 \\ \hline \phantom{\Big|}
	-M_{1C}\la - M_{0C} \phantom{\Big|}& M_{1D}\la + M_{0D} \\
	0 & K_{1D}\la + K_{0D}
	\end{array}
	\right].
	\end{equation}
	Consider the unimodular matrix $\widetilde{U}_A(\lambda)^{-1}=\begin{bmatrix} \widetilde{N}_A(\lambda)^T & \rev_{\rho_A} N_A(\lambda)^T \end{bmatrix}$ defined in \eqref{eq:embedding2}. 
	We have
	\[
	\begin{bmatrix}
	\rev_1 M_A(0)  \\
	\rev_1 K_A(0) \\
	-\rev_1 M_C(0) \\
	0
	\end{bmatrix}\widetilde{U}_A(0)^{-1} = 
	\begin{bmatrix}
	M_{1A}  \\
	K_{1A} \\
	-M_{1C} \\
	0
	\end{bmatrix}\widetilde{U}_A(0)^{-1} = 
	\begin{bmatrix}
	* & \rev_{\rho_A+1}A(0) \\
	I_{n\,\rho_A} & 0 \\
	* & \rev_{\rho_A+1} C(0) \\
	0 & 0
	\end{bmatrix}.
	\]
	The above equation, together with \eqref{eq:right min ind 2}, implies that the matrix
	\[
	\begin{bmatrix}
	M_{1A}  \\
	K_{1A} \\
	-M_{1C} \\
	0
	\end{bmatrix}
	\]
	has full column rank.
	Moreover, from Lemma \ref{Lemma:degrees}, we obtain
	\[
	z_i(\lambda) = 
	\begin{bmatrix}
	y_i(\lambda) \\
	x_i(\lambda)
	\end{bmatrix} = 
	\begin{bmatrix}
	y_{\epsilon_i} \\
	x_{\epsilon_i}
	\end{bmatrix}\lambda^{\epsilon_i} + \mbox{ lower order terms, with }x_{\epsilon_i}\neq 0. 
	\]
	Hence, from $\mathcal{L}(\lambda)z_i(\lambda)=0$, we get
	\[
	\begin{bmatrix}
	M_{1A}  \\
	K_{1A} \\
	-M_{1C} \\
	0
	\end{bmatrix}y_{\epsilon_i}+
	\begin{bmatrix}
	M_{1B}  \\
	0 \\
	M_{1D} \\
	K_{1D}
	\end{bmatrix}x_{\epsilon_i}=0,
	\]
	which implies
	\[
	y_{\epsilon_i} = \underbrace{-
		\begin{bmatrix}
		M_{1A}  \\
		K_{1A} \\
		-M_{1C} \\
		0
		\end{bmatrix}^\dagger
		\begin{bmatrix}
		M_{1B}  \\
		0 \\
		M_{1D} \\
		K_{1D}
		\end{bmatrix}}_{=:E}x_{\epsilon_i} \quad (i=1,\hdots,s),
	\]
	where $\dagger$ denotes the pseudoinverse operation.
	Thus, we have $B_{\rm hcd} = \left[\begin{smallmatrix} E\widehat{B}_{\rm hcd} \\ \widehat{B}_{\rm hcd} \end{smallmatrix} \right]$.
	Therefore, the matrix $\widehat{B}_{\rm hcd}$ must have full column rank since, otherwise, the matrix $B_{\rm hcd}$ would not have full column rank.
	
	By Theorem \ref{Thm:charac min bases}, we conclude that $\{x_i(\lambda)\}_{i=1}^s$ is a right minimal basis of $\widehat{R}(\lambda)$.
	Moreover, since $\epsilon_i = \deg z_i(\lambda) = \deg x_i(\lambda)$, for $i=1,\hdots,s$, the right minimal indices of $\widehat{R}(\lambda)$ are equal to the right minimal indices of $\mathcal{L}(\lambda)$. 
	This establishes the first statement in part (a) and in part (b).
	
	Finally, let us prove that $x_i(\lambda)=N_D(\lambda)^Tu_i(\lambda)$, for some right minimal basis $\{u_i(\lambda)\}_{i=1}^s$ of $R(\lambda)$.
	First, from Lemma \ref{Lemma:from R to Rhat}, we get $x_i(\lambda)=N_D(\lambda)^Tu_i(\lambda)$, for $i=1,\hdots,s$, for some basis  $\{u_i(\lambda)\}_{i=1}^s$ of $\mathcal{N}_r(R)$.
	Since $N_D(\lambda)$ is a minimal basis, by \cite[Main Theorem, part 4]{forney}, the vectors $u_i(\lambda)$ must be polynomial vectors.
	
	We will show that the polynomial basis $\{u_i(\lambda)\}_{i=1}^s$ is minimal by using Theorem \ref{Thm:charac min bases}.
	For this purpose, let us define 
	\[
	\widetilde{B}(\lambda) := 
	\begin{bmatrix}
	u_1(\lambda) & \cdots & u_s(\lambda)
	\end{bmatrix}.
	\]
	Clearly, we have $\widehat{B}(\lambda)=N_D(\lambda)^T\widetilde{B}(\lambda)$.
	Since  $\widehat{B}(\lambda_0)$ has full column rank for every $\lambda_0\in\mathbb{F}$, because $\{x_i(\lambda)\}_{i=1}^s$ is a right minimal basis of $\widehat{R}(\lambda)$, we obtain that $\widetilde{B}(\lambda_0)$ must also have full column rank for every $\lambda_0\in\mathbb{F}$.
	Next, let us denote by $\widetilde{B}_{\rm hcd}$ the highest column degree coefficient matrix of $\widetilde{B}$.
	Since $N_D(\lambda)$ is a minimal basis with all its row degrees equal, the highest column degree coefficient of $N_D(\lambda)^T$ is its leading coefficient matrix, which we denote by $N_{\rho_D}$.
	Then, we have $\widehat{B}_{\rm hcd} = N_{\rho_D}\widetilde{B}_{\rm hcd}$.
	Since $\widehat{B}_{\rm hcd}$ has full column rank, so does $\widetilde{B}_{\rm hcd}$.
	
	By Theorem \ref{Thm:charac min bases}, we conclude that $\{u_i(\lambda)\}_{i=1}^s$ is a right minimal basis of $R(\lambda)$.
	Moreover, by \cite[Main Theorem, part 5]{forney} and the fact that $N_D(\lambda)$ is a minimal basis with all its row degrees equal to $\rho_D$, we have
	\[
	\deg x_i(\lambda) = \rho_D + \deg u_i(\lambda) \quad (i=1,\hdots,s), 
	\]
	which shows that the right minimal indices of $R(\lambda)$ are equal to $\epsilon_1-\rho_D\leq \cdots \leq \epsilon_t-\rho_D$.
	This concludes the proof.
\end{proof}

\section{Proof of Theorem \ref{thm:left_minimalindices}}\label{app:left}

\begin{proof}
	Let us consider a left minimal basis $\{z_i(\lambda)^T\}_{i=1}^t$ of $\mathcal{L}(\lambda)$ with left minimal indices $\eta_i = \deg z_i(\lambda)$, for $i=1,\hdots, t$.
	By Lemma \ref{Lemma: from R to P}, the polynomial vector $z_i(\lambda)^T$ must be of the form
	\[
	z_i(\lambda)^T = 
	\begin{bmatrix}
	y_i(\lambda)^T & x_i(\lambda)^T 
	\end{bmatrix} \quad (i=1,\hdots,t),
	\]
	for some basis $\{x_i(\lambda)^T\}_{i=1}^t$ of the left nullspace of $\widehat{R}(\lambda)$.
	We observe the vectors $x_i(\lambda)^T$ must be polynomial vectors because the vectors $z_i(\lambda)^T$ are polynomial.
	
	We will prove that the polynomial basis $\{x_i(\lambda)^T\}_{i=1}^t$ is a minimal basis by using the characterization in Theorem \ref{Thm:charac min bases}.
	With this goal in mind, let us introduce the polynomial matrices
	\[
	X(\lambda):=
	\begin{bmatrix}
	x_1(\lambda)^T \\
	\vdots \\
	x_t(\lambda)^T
	\end{bmatrix} \quad \mbox{and} \quad
	Z(\lambda):=
	\begin{bmatrix}
	z_1(\lambda)^T \\
	\vdots \\
	z_t(\lambda)^T
	\end{bmatrix} =
	\begin{bmatrix}
	y_1(\lambda)^T & x_1(\lambda)^T \\
	\vdots & \vdots \\
	y_t(\lambda)^T & x_t(\lambda)^T
	\end{bmatrix}  =:
	\begin{bmatrix}
	Y(\lambda) & X(\lambda)
	\end{bmatrix}.
	\]
	
	Let us show, first, that the polynomial matrix $X(\lambda)$ has full row rank for every $\lambda_0\in\mathbb{F}$.
	For this purpose, consider the unimodular matrices $U_A(\lambda)^{-1}$ and $U_D(\lambda)^{-1}$ defined in \eqref{eq:embedding}.
	We have that the matrix
	\[
	\begin{bmatrix}
	M_A(\lambda_0) & M_B(\lambda_0) \\
	K_A(\lambda_0) & 0 
	\end{bmatrix}
	\begin{bmatrix}
	U_A(\lambda_0)^{-1}& 0 \\
	0 & U_D(\lambda_0)^{-1}
	\end{bmatrix} =
	\begin{bmatrix}
	* & A(\lambda_0) & * & B(\lambda_0) \\
	I_{n\,\rho_A} & 0 & 0 & 0 
	\end{bmatrix}
	\]
	has full row rank for every $\lambda_0\in\mathbb{F}$ because of \eqref{eq:left min ind 1}.
	Hence,  $\left[\begin{smallmatrix} M_A(\lambda_0) & M_B(\lambda_0) \\ K_A(\lambda_0) & 0 \end{smallmatrix} \right]$ has full row rank for every $\lambda_0\in\mathbb{F}$.
	Then, from $Z(\lambda_0) \mathcal{L}(\lambda_0)=0$, we obtain
	\[
	\begin{bmatrix}
	Y(\lambda_0) & X(\lambda_0)
	\end{bmatrix}
	\begin{bmatrix}
	M_A(\lambda_0) & M_B(\lambda_0) \\
	K_A(\lambda_0) & 0 \\ \hline
	-M_C(\lambda_0) & M_D(\lambda_0)\\
	0 & K_D(\lambda)
	\end{bmatrix} = 0.
	\]
	By Lemma \ref{Lemma:ranks}, we conclude that $X(\lambda_0)$ has full row rank for every $\lambda_0\in\mathbb{F}$.
	
	Let $Z_{\rm hrd}$ and $X_{\rm hrd}$ be  the highest row degree matrix coefficients of $Z(\lambda)$ and $X(\lambda)$, respectively.
	Let us show that the matrix $X_{\rm hrd}$ has full row rank.
	Consider the unimodular matrices $\widetilde{U}_A(\lambda)^{-1}$ and $\widetilde{U}_D(\lambda)^{-1}$ defined in \eqref{eq:embedding2}.
	We have that
	\begin{align*}
	&\begin{bmatrix}
	\rev_1 M_A(0) & \rev_1 M_B(0) \\
	\rev_1 K_A(0) & 0
	\end{bmatrix}
	\begin{bmatrix}
	\widetilde{U}_A(0)^{-1} & 0 \\
	0 & \widetilde{U}_D(0)^{-1}
	\end{bmatrix} =  \\
	&\hspace{6cm}\begin{bmatrix}
	* & \rev_{\rho_A+1}A(0) & * & \rev_{\rho_D+1}B(0) \\
	I_{n\, \rho_A} & 0 & 0 & 0
	\end{bmatrix}
	\end{align*}
	has full row rank for every $\lambda_0\in\mathbb{F}$ because of \eqref{eq:left min ind 2}.
	Hence, using the notation introduced in \eqref{eq:notation L}, we have that the matrix $\left[\begin{smallmatrix} \rev_1 M_A(0) & \rev_1 M_B(0) \\ \rev_1 K_A(0) & 0 \end{smallmatrix} \right]=\left[\begin{smallmatrix} M_{1A} & M_{1B} \\ K_{1A} & 0 \end{smallmatrix} \right]$ has full row rank for every $\lambda_0\in\mathbb{F}$ .
	Thus, Lemma \ref{Lemma:degrees} implies
	\[
	z_i(\lambda)^T = 
	\begin{bmatrix}
	y_i(\lambda)^T & x_i(\lambda)^T 
	\end{bmatrix} = 
	\begin{bmatrix}
	y_{\eta_i}^T & x_{\eta_i}^T
	\end{bmatrix}\lambda^{\eta_i} + \mbox{ lower order terms, with }x_{\eta_i}\neq 0.
	\]
	Then, from  $z_i(\lambda)^T \mathcal{L}(\lambda)=0$, we obtain
	\[
	\begin{bmatrix}
	y_{\eta_i}^T & x_{\eta_i}^T
	\end{bmatrix}
	\begin{bmatrix}
	M_{1A} & M_{1B} \\
	K_{1A} & 0 \\ \hline
	-M_{1C} & M_{1D} \\
	0 & K_{1D}
	\end{bmatrix}=0.
	\]
	Since the matrix $\left[\begin{smallmatrix} M_{1A} & M_{1B} \\ K_{1A} & 0 \end{smallmatrix} \right]$ has full row rank, we have
	\[
	y_{\eta_i}^T = -x_{\eta_i}^T \underbrace{
		\begin{bmatrix}
		-M_{1C} & M_{1D} \\
		0 & K_{1D}
		\end{bmatrix} 
		\begin{bmatrix}
		M_{1A} & M_{1B} \\
		K_{1A} & 0
		\end{bmatrix}^\dagger}_{=:F},
	\]
	where $\dagger$ indicates the pseudoinverse operation. 
	Therefore, $Z_{\rm hrd} = \begin{bmatrix} X_{\rm hrd}F & X_{\rm hrd} \end{bmatrix}$.
	Conclusively, the matrix $X_{\rm hrd}$ has full row rank because $Z_{\rm hrd}$ has full row rank.
	
	From Theorem \ref{Thm:charac min bases}, we get that $\{x_i(\lambda)^T\}_{i=1}^t$ is a left minimal basis of $\widehat{R}(\lambda)$.
	Moreover, by Lemma \ref{Lemma:degrees}, we have $\eta_i = \deg z_i(\lambda) = \deg x_i(\lambda)$, for $i=1,\hdots,t$.
	Therefore, $\mathcal{L}(\lambda)$ and $\widehat{R}(\lambda)$ have the same left minimal indices.
	This establishes the first statement in part (a) and in part (b).
	
	By Lemma \ref{Lemma:from R to Rhat}, the vector $x_i(\lambda)^T$ must be of the form
	\[
	x_i(\lambda)^T = \begin{bmatrix}
	u_i(\lambda)^T & w_i(\lambda)^T 
	\end{bmatrix}\quad (i=1,\hdots,t), 
	\]
	for some basis $\{u_i(\lambda)^T\}_{i=1}^t$ of the left nullspace of $R(\lambda)$.
	We notice that the vectors $u_i(\lambda)$ must be polynomial vectors because the $x_i(\lambda)$ vectors are polynomial.
	Our final goals are, first, to show that $\{u_i(\lambda)^T\}_{i=1}^T$ is a left minimal basis of $R(\lambda)$ and, second, to show that $\deg u_i(\lambda) = \eta_i$, for $i=1,\hdots, t$.
	
	We begin by noticing that if we combine Lemmas \ref{Lemma: from R to P} and \ref{Lemma:from R to Rhat}, we get that the vector $z_i(\lambda)^T$ must be of the form
	\[
	z_i(\lambda)^T = 
	\begin{bmatrix}
	\alpha_{1i}(\lambda)^T & \alpha_{2i}(\lambda)^T & u_i(\lambda)^T & \alpha_{3i}(\lambda)^T
	\end{bmatrix} \quad (i=1,\hdots,t),
	\]
	for some vectors $\alpha_{ji}(\lambda)$, with $j=1,2,3$, conformable with the partition of $\mathcal{L}(\lambda)$.
	We claim that $\{\begin{bmatrix} \alpha_{1i}(\lambda)^T & u_i(\lambda)^T \end{bmatrix} \}_{i=1}^t$ is a left minimal basis of the polynomial system matrix $P(\lambda) = \left[\begin{smallmatrix} A(\lambda) & B(\lambda) \\ -C(\lambda) & D(\lambda) \end{smallmatrix}\right]$ with left minimal indices $\eta_1\leq \cdots\leq \eta_t$.
	To see this, first, from $z_i(\lambda)^T\mathcal{L}(\lambda)=0$, we obtain
	\begin{equation}\label{eq:aux1}
	\begin{bmatrix}
	\alpha_{1i}(\lambda)^T & u_i(\lambda)^T 
	\end{bmatrix}
	\begin{bmatrix}
	M_A(\lambda) & M_B(\lambda) \\
	-M_C(\lambda) & M_D(\lambda) 
	\end{bmatrix} +
	\begin{bmatrix}
	\alpha_{2i}(\lambda)^T & \alpha_{3i}(\lambda)^T
	\end{bmatrix}
	\begin{bmatrix}
	K_A(\lambda) & 0 \\
	0 & K_D(\lambda)
	\end{bmatrix} = 0.
	\end{equation}
	Multiplying \eqref{eq:aux1} on the right by $\diag(N_A(\lambda)^T,N_D(\lambda)^T)$ yields
	\begin{equation}\label{eq:leftvector_polysystem}
	\begin{bmatrix}
	\alpha_{1i}(\lambda)^T & u_i(\lambda)^T 
	\end{bmatrix}
	\begin{bmatrix}
	A(\lambda) & B(\lambda) \\
	-C(\lambda) & D(\lambda)
	\end{bmatrix}=0.
	\end{equation}
	Hence, the polynomial vector $\begin{bmatrix} \alpha_{1i}(\lambda)^T & u_i(\lambda)^T \end{bmatrix}\in\mathcal{N}_\ell(P)$.

	Next, let us consider the polynomial matrix
	\[
	\widehat{U}(\lambda) :=
	\begin{bmatrix}
	\alpha_{11}(\lambda)^T & u_1(\lambda)^T \\
	\vdots & \vdots \\
	\alpha_{1t}(\lambda)^T & u_t(\lambda)^T
	\end{bmatrix} = : 
	\begin{bmatrix}
	A_1(\lambda) & U(\lambda)
	\end{bmatrix}.
	\]
	From \eqref{eq:leftvector_polysystem}, $\widehat{U}(\lambda)P(\la)=0$. Let us show that $\widehat{U}(\lambda_0)$ has full row rank for every $\lambda_0\in\mathbb{F}$.
	From $Z(\lambda_0)\mathcal{L}(\lambda_0)=0$, we obtain
	\[
	\begin{bmatrix}
	\widehat{U}(\lambda_0) & *
	\end{bmatrix} 
	\begin{bmatrix}
	M_A(\lambda_0) & M_B(\lambda_0) \\
	-M_C(\lambda_0) & M_D(\lambda_0) \\ \hline
	K_A(\lambda_0) & 0 \\
	0 & K_D(\lambda_0)
	\end{bmatrix} = 0.
	\]
	Since $\diag(K_A(\lambda_0),K_D(\lambda_0))$ has full row rank for every $\lambda_0\in\mathbb{F}$ (because $K_A(\lambda)$ and $K_D(\lambda)$ are both minimal basis), we conclude, by Lemma \ref{Lemma:ranks}, that the matrix $\widehat{U}(\lambda_0)$ has full row rank for every $\lambda_0\in\mathbb{F}$.
	
	Let us denote by $\widehat{U}_{\rm hrd}$ the highest row degree coefficient matrix of $U_{\rm hrd}$, and  show that  $\widehat{U}_{\rm hrd}$  has full row rank.
	Since $K_A(\lambda)$ and $K_D(\lambda)$ are minimal bases with constant row degrees (equal to 1), by \cite[Main Theorem, part 5]{forney}, we get from \eqref{eq:aux1}
	\begin{align*}
	\deg 
	\begin{bmatrix}
	\alpha_{2i}(\lambda)^T & \alpha_{3i}(\lambda)^T
	\end{bmatrix}
	\begin{bmatrix}
	K_A(\lambda) & 0 \\
	0 & K_D(\lambda)
	\end{bmatrix} = &
	1 + \deg
	\begin{bmatrix}
	\alpha_{2i}(\lambda)^T & \alpha_{3i}(\lambda)^T
	\end{bmatrix} \leq \\
	&1 + \deg
	\begin{bmatrix}
	\alpha_{1i}(\lambda)^T & u_i(\lambda)^T 
	\end{bmatrix}.
	\end{align*}
	Hence, we have
	\begin{equation}\label{eq:degrees}
	\deg
	\begin{bmatrix}
	\alpha_{2i}(\lambda)^T & \alpha_{3i}(\lambda)^T
	\end{bmatrix} \leq
	\deg 
	\begin{bmatrix}
	\alpha_{1i}(\lambda)^T & u_i(\lambda)^T 
	\end{bmatrix}.
	\end{equation}
	Thus, from $Z(\lambda)\mathcal{L}(\lambda)=0$, we obtain
	\[
	\begin{bmatrix}
	\widehat{U}_{\rm hrd} & * 
	\end{bmatrix}
	\begin{bmatrix}
	M_{1A} & M_{1B} \\
	-M_{1C} & M_{1D} \\ \hline
	K_{1A} & 0 \\
	0 & K_{1D}
	\end{bmatrix}=0.
	\]
	Since the matrix $\diag(K_{1A},K_{1D})$ has full row rank (because $K_A(\lambda)$ and $K_D(\lambda)$ are both minimal bases with constant row degrees), by Lemma \ref{Lemma:ranks}, we have that the matrix $\widehat{U}_{\rm hrd}$ has also full row rank.
	
	Conclusively, since $\dim \mathcal{N}_\ell(\mathcal{L})=\dim \mathcal{N}_\ell(P)$ (becasue $\mathcal{L}(\lambda)$ is a linearization of $P(\lambda)$ in the classical matrix polynomial sense), Theorem \ref{Thm:charac min bases} implies that 
	$\{\begin{bmatrix} \alpha_{1i}(\lambda)^T & u_i(\lambda)^T \end{bmatrix} \}_{i=1}^T$ is a left minimal basis of $P(\lambda)$, as we claimed.
	Moreover, by $\eqref{eq:degrees}$, we have $\eta_i = \deg z_i(\lambda)^T = \deg \begin{bmatrix} \alpha_{1i}(\lambda)^T & u_i(\lambda)^T  \end{bmatrix}$.
	Hence, the left minimal indices of $P(\lambda)$ are also equal to $\eta_1\leq \cdots \leq \eta_t$.
	
	After this small detour, we are ready to prove that $\{u_i(\lambda)^T\}_{i=1}^t$ is a minimal basis of $R(\lambda)$, with minimal indices $\eta_1\leq \cdots\leq \eta_t$, by using Theorem \ref{Thm:charac min bases}.
	To this goal, let us consider the polynomial matrix
	\[
	U(\lambda) :=
	\begin{bmatrix}
	u_1(\lambda)^T \\
	\vdots\\
	u_t(\lambda)^T
	\end{bmatrix}.
	\]
	From $\widehat{U}(\lambda_0)P(\lambda_0)=0$, we get
	\[
	\begin{bmatrix}
	A_1(\lambda_0) & U(\lambda_0)  
	\end{bmatrix}
	\begin{bmatrix}
	A(\lambda_0) & B(\lambda_0) \\ 
	-C(\lambda_0) & D(\lambda_0)
	\end{bmatrix}=0.
	\] 
	Since the matrix $\begin{bmatrix} A(\lambda_0) & B(\lambda_0) \end{bmatrix}$ has full row rank for every $\lambda_0\in\mathbb{F}$, Lemma \ref{Lemma:ranks} implies that $U(\lambda_0)$ has also full row rank for every $\lambda_0\in\mathbb{F}$.
	
	Let $U_{\rm hrd}$ denote the highest row degree coefficient matrix of $U(\lambda)$.
	It remains to show that $U_{\rm hrd}$ has full row rank. First, let us prove that
		\begin{equation}\label{eq:aux degrees}
		\deg \begin{bmatrix} \alpha_{1i}(\lambda)^T & u_i(\lambda)^T \end{bmatrix} = \deg u_i(\lambda)^T \quad (i=1,\hdots,t).
		\end{equation}
		By contradiction, assume  \begin{equation*}\label{eq:aux degrees2}
		\begin{bmatrix} \alpha_{1i}(\lambda)^T & u_i(\lambda)^T \end{bmatrix} = \begin{bmatrix}\alpha_i^{T} & 0\end{bmatrix}\la^{\eta_i} + \mbox{ lower order terms, with }\alpha_i\neq 0.
		\end{equation*}
		Since $\begin{bmatrix} \alpha_{1i}(\lambda)^T & u_i(\lambda)^T \end{bmatrix} P(\lambda)=0$, we have that $\alpha_{i}^{T}\begin{bmatrix}\rev_{\rho_A +1}A(0) & \rev_{\rho_D +1}B(0) \end{bmatrix}=0$. Then $\alpha_i=0$ by \eqref{eq:left min ind 2}, which is a contradiction. Hence $\widehat{U}_{\rm hrd}$ must be of the form $\widehat{U}_{\rm hrd} = 
	\begin{bmatrix}
	\mathcal{A} & U_{\rm hrd}
	\end{bmatrix}$,
	for some matrix $\mathcal{A}$. Considering again that $\widehat{U}(\lambda)P(\lambda)=0$, we have  $$\begin{bmatrix}
	A_1(\lambda) & U(\lambda)  
	\end{bmatrix}
	\begin{bmatrix}
	A(\lambda)  \\ 
	-C(\lambda) 
	\end{bmatrix}=0 \text{ and }\begin{bmatrix}
	A_1(\lambda) & U(\lambda)  
	\end{bmatrix}
	\begin{bmatrix}
	B(\lambda)  \\ 
	D(\lambda) 
	\end{bmatrix}=0.$$ Therefore,
	\[
	\begin{bmatrix}
	\mathcal{A}  & U_{\rm hrd}
	\end{bmatrix} 
	\begin{bmatrix}
	\rev_{\rho_A +1} A(0) & \rev_{\rho_D +1} B(0) \\
	-\rev_{\rho_A +1} C(0) & \rev_{\rho_D +1} D(0)
	\end{bmatrix}=0.
	\]
	Taking into account condition \eqref{eq:left min ind 2}, Lemma \ref{Lemma:ranks} implies that $U_{\rm hrd}$ has full row rank. 
	
	Thus, by Theorem \ref{Thm:charac min bases}, $\{u_i(\lambda)^T\}_{i=1}^T$ is a left minimal basis of $R(\lambda)$.
	Moreover, by \eqref{eq:aux degrees}, the left minimal indices of $R(\lambda)$ are $\eta_1\leq \cdots \leq \eta_t$, as we wanted to prove.
\end{proof}

\bibliographystyle{plain}

\end{document}